\documentclass[11pt]{amsart}      

\usepackage{amssymb}      
\usepackage{amsmath}      
\usepackage{amsthm}      
\usepackage{amsbsy}       
\usepackage{amscd} 
\usepackage{latexsym} 
\usepackage[dvips]{graphicx}

\usepackage{bm}    

\usepackage[
	hypertexnames=false,
	hyperindex,
	pagebackref,  
	bookmarks=false,
	colorlinks,
	linkcolor=blue,
	citecolor=red,
	urlcolor=red,
]{hyperref}

\theoremstyle{plain}      
\newtheorem{theorem}{Theorem}[section]      
\newtheorem{lemma}{Lemma}[section]      
\newtheorem{corollary}[theorem]{Corollary}      
\newtheorem{proposition}{Proposition}[section]      
    
 \newtheorem{question}{Question}[section]       
      
\newtheorem{definition}{Definition}[section]          
\theoremstyle{remark}      
\newtheorem{remark}{Remark}[section]

\newcommand{\Q}{{\mathbb{Q}}}        
\newcommand{\Z}{{\mathbb{Z}}}   
   
\newcommand{\C}{{\mathbb{C}}}      
\newcommand{\R}{{\mathbb{R}}}

  \newcommand{\Aut}{{\rm{Aut}}}  
  \newcommand{\Out}{{\rm{Out}}}  

\newcommand{\G}{{\Gamma}}
\newcommand{\Mod}{{\frak{M}}}

\renewcommand{\O}{{\mathcal{O}}}

  \newenvironment{dedication}
  {\thispagestyle{empty}
   \itshape             
   \raggedleft          
  }
  {\par 
  }

\begin{document}


\title[Mapping class group quotients]{On mapping class group quotients by powers of Dehn twists and their representations    
}     

 \author{Louis Funar} 
 \address{ Institut Fourier, UMR 5582, Laboratoire de Math\'ematiques \\
Universit\'e Grenoble Alpes, CS 40700, 38058 Grenoble cedex 9, France}

\maketitle 
\begin{dedication}
Dedicated to Vladimir G. Turaev on the occasion of his 65-th birthday
\end{dedication}

\begin{abstract}
The aim of this paper is to survey some known results about mapping class group 
quotients by powers of Dehn twists, related to their finite dimensional 
representations and to state some open questions. One can construct finite quotients of them, 
out of representations with Zariski dense images into semisimple Lie groups. We show that, in genus 2, the Fibonacci TQFT representation is actually a specialization of the Jones representation.  Eventually, we explain a method of Long and Moody which provides    large families of mapping class group representations.

\vspace{0.1cm}
 
\end{abstract}

\section{Mapping class group quotients}
\subsection{Introduction}

Set $\Sigma_{g,k}^r$ for the 
orientable surface of genus $g$ with $k$ boundary components and $r$ 
marked points. We denote by 
$\Gamma_{g,k}^r$ the  mapping class group of $\Sigma_{g,k}^r$, namely 
the group of isotopy classes of orientation-preserving homeomorphisms that fix pointwise 
the boundary components and preserve globally the set of marked points. The pure mapping class group $P\Gamma_{g,k}^r$ 
consists of those classes of homeomorphisms which fix pointwise both the 
boundary components and each of the marked points. 
 
We set $\pi_{g,k}^r$ for the fundamental group of the surface $\Sigma_{g,k}^r$. 
Recall that, by the Dehn--Nielsen--Baer  theorem $\G_g^1$ is the group of orientation-preserving automorphisms  of 
$\pi_g$, namely those which preserve the conjugacy class of the relator instead of reversing it.   
Further $\G_g={\rm Out}^+(\pi_g)={\rm Aut}^+(\pi_g)/{\rm Inn}(\pi_g)$, where ${\rm Inn}(\pi_g)$ is the subgroup of 
inner automorphisms of $\pi_g$. 
There is a more general identification of algebraic and topological mapping class groups, as follows. Denote by $\gamma_j$ and $\delta_s$ the loops around the punctures and respectively the boundary components and 
by $[z]$ the conjugacy class in $\pi_{g,k}^r$ of the element $z$. Let ${\rm Aut}^+(\pi_{g,k}^r; C_1,\ldots,C_s)$ stands for the subgroup of those automorphisms fixing globally each set of conjugacy classes in $C_1, C_2, \ldots,C_s$. 
Let $P_{r}$ be the set of all peripheral conjugacy classes $[\gamma_j]$ and  
$\mathbf P_r$ be the vector consisting of these peripheral conjugacy classes. Similarly, $P^{\partial}_k$ is the set 
of all boundary conjugacy classes $[\delta_j]$ and by 
$\mathbf P^{\partial}_k$ the vector consisting of these peripheral conjugacy classes.
Then the Dehn--Nielsen--Baer  theorem states that there is an isomorphism: 
\[ \Gamma_{g,k}^r\simeq {\rm Out}^+(\pi_{g,k}^r, P_{r}, \mathbf P^{\partial}_k)= {\rm Aut}^+(\pi_{g,k}^r; P_{r}, \mathbf P^{\partial}_k))/{\rm Inn}(\pi_{g,k}^r).\]
The notation is intended to specify that each boundary conjugacy class is fixed, while the peripheral conjugacy classes are only globally invariant.  
If we fix the base point of the fundamental group to be among the marked points, it will be automatically 
invariant by the pure mapping class group so that:  
\[ P\Gamma_{g,k}^{r+1} \simeq {\rm Aut}^+(\pi_{g,k}^r; \mathbf P_{r}, \mathbf P^{\partial}_k).\]

The main questions addressed here concern the  finite-dimensional 
representations of mapping class groups.  We shall motivate the introduction 
of the normal subgroups $\Gamma_{g,k}^p[p]$ generated by  
$p$-th powers $T_{\gamma}^p$ of Dehn twists $T_{\gamma}$ along simple curves 
$\gamma$ on the surface. Furthermore, we introduce the family of characteristic 
quotients 
\[ \Gamma_{g,k}^r(p)=\Gamma_{g,k}^r/\Gamma_{g,k}^r[p].\]
More generally, if $G\subseteq \Gamma_{g,k}^r$ is a subgroup, then we denote 
by $G(p)$ the image of $G$ within the quotient $\Gamma_{g,k}^r(p)$. 

In the first half of this article we survey some of the known properties  and  state 
some questions concerning the groups $\Gamma_{g,k}^r(p)$, in relation with  their 
representations. The main source of finite-dimensional representations for these groups 
are the modular tensor categories (see \cite{Turaev}), which arose from the seminal work of 
Witten (\cite{Witten}), Reshetikhin and Turaev (\cite{RT}) on 3-manifold invariants. 
We discuss some properties 
of the family of representations associated with the groups $SU(2)/SO(3)$. 
In the last part we explain some algebraic/geometric constructions of 
mapping class group representations following Long and Moody.

\subsection{Compact representations of mapping class groups}
The quotients $\Gamma_{g,k}^r(p)$ arise naturally when we study representations of mapping class groups into compact Lie groups. 

\begin{definition}
A representation $\rho:\Gamma_{g,k}^r\to G$ into a linear algebraic group $G$
is called {\em unipotent-free} if the images of the Dehn twists are diagonalizable elements 
in $G_{\C}$. 
\end{definition}
\noindent For instance, representations into a compact Lie group $G$ are automatically unipotent-free representations.

\begin{proposition}\label{unipotentfree}
Let $G$ be a linear algebraic group. There exists some $p=p(G)$, such that  
any  unipotent-free representation  $\rho:\Gamma_{g,k}^r\to G$, for $g\geq 3$ 
factors through the quotient $\Gamma_{g,k}^r(p)$. 
\end{proposition}
\begin{proof}
The proof given by  Aramayona--Souto (\cite{AS}) in the case where $\rho(\Gamma_{g,k}^r)$ 
contains no unipotent elements actually is valid for all unipotent-free representations. 
We sketch below the main argument, for the sake of completeness. 
Let $\rho:\Gamma_{g,k}^r\to G_{\C}\subset GL(V)$ be a unipotent-free representation. 
If $\gamma$ is a simple curve on a $\Sigma_{g,k}^r$ then at least one component of 
the surface $S$ obtained by cutting along $\gamma$ has genus larger than or equal to $2$.
Note further that the Dehn twist $T_{\gamma}$ along a boundary curve $\gamma$ of 
the surface $S$ belongs to the center of the mapping class groups $\Gamma(S)$.
If $W$ is an eigenspace for $\rho(T_{\gamma})$, then $W$ must be invariant by 
$\rho(\Gamma(S))$. Therefore we obtain a homomorphism $\Gamma(S)\to \C^*$ 
sending $x\in \Gamma(S)$ into $\det(\rho(x)|_{W})$.  Since $H_1(\Gamma(S),\Q)=0$ 
we derive that this homomorphism has image contained in the group of roots of unity of
order 10. In particular, eigenvalues of $\rho(T_{\gamma})$ are roots of unity, whose order is bounded in terms of $\dim V$ alone. Since  $\rho(T_{\gamma})$ is diagonalizable, 
it has finite order dividing some $p$ which only depends on $\dim V$. 

Alternatively, this follows from Bridson's Theorem 2 from \cite{Br} which states 
if $\Gamma_{g,k}^r$ acts by isometries on complete CAT(0) spaces then Dehn twist act either as elliptics or neutral parabolics, the second case being prohibited by the hypothesis. 
\end{proof}

Thus the study of unipotent-free representations of mapping class groups pops out the need of understanding the quotients $\Gamma_{g,k}^r(p)$. The simplest constructions of 
finite-dimensional mapping class group representations yield parabolic matrices for 
Dehn twists, as it is the case for the homological ones. The first interesting examples 
of unitary representations arise in topological quantum field theory, by means of the 
methods pioneered by Reshetikhin and Turaev and further Turaev and Viro. Later one 
observed that various classical constructions, as the Burau representations of braid groups
and homological representations of coverings, as studied by Looijenga (\cite{Loo}), can lead to unipotent-free representations.  Note that, generically, the groups 
 $\Gamma_{g,k}^r(p)$ are infinite (see \cite{F}). 
Recently, new methods from combinatorial group theory permitted to prove that 
for any $g,r$ there exists some $k_0$ such that $\Gamma_{g}^r(p)$ are acylindrically hyperbolic 
(see \cite{DHS}) and also hierarchically hyperbolic 
(see \cite{BHMS}), when $k_0$ divides $p$.

\subsection{Finite index subgroups of $\Gamma_g(p)$}
Let $\mathcal I_{g,k}$ be the $k$-th Johnson subgroup of $\Gamma_g$; in particular   
$\mathcal I_{g,1}=\mathcal T_g$ is the Torelli group and
$\mathcal I_{g,2}=\mathcal K_g$ is the Johnson kernel. 
The following  lemma was already used in \cite{F13} and we record it here for further use:  

\begin{lemma}\label{fitorelli}
The Torelli group quotient $\mathcal T_{g,k}^r(p)\subseteq \Gamma_{g,k}^r(p)$
has finite index. 
\end{lemma}
\begin{proof}
Dehn twists have finite order dividing $p$ in the image. 
Recall that $\Gamma_{g,k}^r$ is generated by Dehn twists and braids (half-Dehn twists).
Dehn twists have finite order dividing $p$ in $\Gamma_{g,k}^r(p)$. 
Now Dehn twists act as transvections on $H_1(\Sigma_{k,k}^r;\Z)$ and they generate 
the group $\Aut^*(H_1(\Sigma_{k,k}^r;\Z))$ of automorphisms preserving the 
intersection form. Then $p$-th powers of Dehn twists generate the 
congruence subgroup 
\[\ker(\Aut^*(H_1(\Sigma_{k,k}^r;\Z))\to \Aut^*(H_1(\Sigma_{k,k}^r;\Z/p\Z))).\]
The exact sequence 
\[ 1\to \mathcal T_{g,k}^r \to \Gamma_{g,k}^r\to \Aut^*(H_1(\Sigma_{k,k}^r;\Z))\]
induces an exact sequence 
\[1\to \mathcal T_{g,k}^r(p)\to \Gamma_{g,k}^r(p)\to \Aut^*(H_1(\Sigma_{k,k}^r;\Z/p\Z)).\]
This implies that  $T_{g,k}^r(p)$ is a finite index normal subgroup of $\Gamma_{g,k}^r(p)$. 
\end{proof}

We need the following well-known lemma concerning nilpotent groups (see e.g. \cite{Clement}, Corollary 2.10): 

\begin{lemma}\label{finitenilpotent}
Let $N$ be a nilpotent group. Then $N$ is finite if and only if $H_1(N)$ is finite. 
\end{lemma}

If $G$ is a group we denote by $\gamma_kG$ the lower central series of $G$, $\gamma_1(G)=G$ and 
$\gamma_{k+1}(G)=[G,\gamma_k(G)]$.
For the sake of simplicity we restrict now to the case of closed surfaces. 

\begin{proposition}\label{lcseries}
For every $k\geq 1$ and $g\geq 3$, the lower central series terms 
$\gamma_k \mathcal T_g(p)$ have finite index in $\mathcal T_g(p)$. 
In particular the  images $\mathcal I_{g,k}(p)$ of the higher Johnson subgroups $\mathcal I_{g,k}\subset \mathcal T_g$  are also finite index subgroups of $\Gamma_g(p)$. 
\end{proposition}
\begin{proof}
The Torelli group $\mathcal T_g$ is generated by BP pairs, when $g\geq 2$.
Lemma \ref{fitorelli} implies that $\mathcal T_g(p)$ is finitely generated, because 
it is a finite index subgroup of the finitely generated group $\Gamma_g(p)$. 
In particular  $H_1({\mathcal T_g(p)})$ is of finite type. 
Since  images of BP pairs into $\Gamma_g(p)$ have order $p$, we derive that 
the group $H_1({\mathcal T_g(p)})$ has a generating set consisting of elements of order $p$ 
and hence is finite.
Thus $\mathcal T_g(p)/\gamma_k \mathcal T_g(p)$ is a finitely generated nilpotent group  whose 
abelianization is finite. It is therefore finite, by Lemma \ref{finitenilpotent}. 

Further, recall that the Johnson filtration is a central descending series with torsion-free quotients (see \cite{Hain}, Prop. 14.5) and in particular 
\[ \gamma_k(\mathcal T_g)\subseteq \mathcal I_{g,k}.\]
Now $\mathcal T_g/\gamma_k\mathcal T_g$ surjects onto 
$\mathcal T_g/\mathcal I_{g,k}$ and hence the latter is a nilpotent group. 
This implies that $\mathcal T_g(p)/\mathcal I_{g,k}(p)$ is also a nilpotent group. 
Its abelianization is a quotient of $H_1({\mathcal T_g(p)})$ and hence $\mathcal T_g(p)/\mathcal I_{g,k}(p)$ is finite by 
Lemma \ref{finitenilpotent}. 
\end{proof}

We will show later that the images $G(p)$ of reducible subgroups $G$ 
of $\Gamma_{g,k}^r$ are of infinite index in $\Gamma_{g,k}^r(p)$. 
Moreover, there are also irreducible subgroups  $G$ of  $\Gamma_{g,k}^r$ which 
not virtually abelian, such that $G(p)$ has infinite index in $\Gamma_{g,k}^r(p)$, as we 
shall see in Section \ref{section:handlebody}.
However, the following question seems relevant: 

\begin{question}
Are there irreducible subgroups $G$ of $\Gamma_{g}^r$ 
which are neither virtually abelian nor virtually conjugate into a subgroup of mapping 
classes which extend to a 3-manifold (e.g. a handlebody group) such that 
$G(p)$ is of infinite index in $\Gamma_{g}^r(p)$ for large enough $p$? 
\end{question}

Propositions  \ref{unipotentfree}  and  \ref{lcseries} imply immediately the following generalization of the result obtained in \cite{F13} for quantum representations: 
\begin{proposition}
Let $\rho:\Gamma_g\to G$ be a unipotent-free representation into a  Lie group $G$. Assume $g\geq 3$. 
Then $\rho(\gamma_k\mathcal T_g)$ and hence also $\rho(\mathcal I_{g,k})$ have finite index in $\rho(\Gamma_g)$. 
\end{proposition}

\subsection{K\"ahler groups and rank 1 representations} 
It was proved in (\cite{AF}, Thm. 5 see also Thm. 15 due to Pikaart and Jong) that: 
\begin{proposition}\label{kahler}
For every $g\geq 2$ the groups $\Gamma_g^r(p)$ are virtually K\"ahler, indeed they have   
finite index subgroups which are fundamental groups of smooth complex projective varieties. In particular, the images $\mathcal I_{g,2}(p)$ of the Johnson kernel are K\"ahler groups, when $p$ is odd.  
\end{proposition}
\begin{remark}
In recent work \cite{EF} the authors were able to show that $\Gamma_g^r(p)$ are 
actually K\"ahler. 
\end{remark}  

An immediate consequence of the alternative proved by Delzant (see \cite{Delzant}) 
for K\"ahler groups states then: 

\begin{proposition}
Either any solvable quotient of any finite index subgroup of $\mathcal I_{g,2}(p)$ 
is virtually nilpotent or else $\mathcal I_{g,2}(p)$  has a finite index subgroup which surjects onto a nonabelian surface group.  
\end{proposition}
It is presently unknown which one of the two alternatives above holds. However, if 
$\Gamma_g$ ($g\geq 3$) does not virtually surject onto $\Z$, then the second alternative 
cannot hold. We expect that the first alternative only could hold when all solvable quotients  
are actually finite. This could be proved if we could promote the virtual nilpotence above to a genuine nilpotence.

\begin{proposition}\label{rank1}
Let $f:\Gamma_g\to \Lambda$ be a homomorphism in a torsion-free uniform rank 1 lattice $\Lambda$ 
in $SO(1,n)$, with $n\geq 3$. 
If $g\geq 3$, then $f$ is trivial, i.e. with finite image. 
\end{proposition}
\begin{proof}
Since $\Lambda$ is cocompact, the homomorphism $f$ is unipotent-free. Therefore 
there exists some $p$ such that $f$ factors through $\Gamma_g(p)$.
Recall that $I_{g,2}(p)$ is K\"ahler and hence the fundamental group 
of some compact K\"ahler manifold $\mathcal X$.  
Then $f|_{I_{g,2}(p)}$ is 
induced by a map $F:\mathcal X\to \mathbb H^{n+1}_{\R}/\Lambda$ into a hyperbolic space form. 
Eells--Sampson \cite{ES} proved that then the map $F$ could be assumed to be 
a harmonic map. A result due to Carlson--Toledo (see  \cite{CT}, Thm. 7.1 and Cor. 3.7) shows that 
a harmonic map as above factors either through a circle or else through a compact Riemann surface. 
Thus $f|_{\mathcal I_{g,2}(p)}$ factors through $\Z$ or through $\pi_1(\Sigma_h)$. 

Now, Dimca and Papadima proved  in \cite{DP} that $H_1(\mathcal I_{g,2})$ is finitely generated and hence $H_1(\mathcal I_{g,2}(p))$ is finite, because it is generated by finitely many classes of 
Dehn twists. We derive that   $f|_{\mathcal I_{g,2}(p)}$ is trivial and hence 
$f$ is trivial. 
\end{proof}

\begin{question}
Are there any nontrivial (i.e. non virtually solvable image) homomorphisms $\Gamma_g\to SO(1,n)$, $n\geq 3$ and $g\geq 3$? 
\end{question}

Proposition \ref{rank1} cannot extend to genus 2, as homomorphisms of $\Gamma_2$ 
do not necessarily factor through $\Gamma_2(p)$. However, there exists a nontrivial 
homomorphism $\Gamma_2(5)\to PU(1,4)$ arising in the Fibonacci TQFT whose image is 
not virtually solvable, which will be explained later.

\section{Quantum representations} 
\subsection{Arithmetic groups and quantum representations}
We now consider the first examples of unitary (hence unipotent-free) representations of 
mapping class groups with infinite image (see \cite{F}). Although we  
use the generic term {\em quantum} representations for a specific family $\rho_p$ 
depending on an integer $p$, one should note that the algebraic machinery of modular tensor categories, in particular quantum groups, provide a large supply of such 
finite dimensional representations (see \cite{Turaev}). A comprehensive introduction to  quantum representations can be found in \cite{Marche}. 
    
Set $U=U(H)$  for the unitary group preserving a
Hermitian form $H$. We suppose that $H$ is associated to a non-degenerate sesquilinear form defined over a  totally real number field $\mathbb K$. Let $\O_K$ be the ring of algebraic 
integers in $\mathbb K$. 
The group of integral points $\widetilde{\Lambda}=SU(\O_{\mathbb K})$ is a lattice in the 
group $S\mathbb U= Res_{\mathbb K/\Q}SU$ obtained by the Weil restriction of scalars 
from $\mathbb K$ to $\Q$. Specifically $S\mathbb U=\prod_{\sigma} SU(f^\sigma)$, where  
$\sigma$ belongs to the set of real places of $\mathbb K$, i.e. embeddings $\sigma:\mathbb K\to \R$, up to conjugacy. 
Let $H_{g,p}$ be the Hermitian form on the space of conformal blocks $W_{g,p}$ in level $p$
and $U_{g,p}=U(H_{g,p})$. We drop $g$ and/or $p$ when irrelevant.

In the case of the SU(2)/SO(3) theory  we have a representation $\tilde{\rho}_p$ of a central extension $\widetilde{\Gamma_{g}}$ of the mapping class group by $\Z$ into the unitary group $U_{g,p}$, defined over a cyclotomic field $\mathbb K_p$. Note that for odd $p$, $\mathbb K_p$ is the totally real cyclotomic field 
$\Q(\zeta_p+\overline{\zeta_p})$, when $p\equiv 3({\rm mod} \; 4)$ and 
 $\Q(\zeta_{4p}+\overline{\zeta_{4p}})$, otherwise. 

Now, a key property of this construction is that the corresponding projective 
representation $\rho_p$ of $\Gamma_g$ factors through $\Gamma_g(p)$, for odd $p$ 
and through $\Gamma_g(2p)$, for  even $p$. This is not surprising in view of 
Proposition \ref{unipotentfree}, since these representations 
are finite-dimensional compact representations. 
 
A deep theorem of Gilmer-Masbaum (\cite{GM}) actually says that 
the image  $\widetilde{\mathcal L_p}=\widetilde{\rho}_{p}(\widetilde{\Gamma_g})\subset SU(\mathbb K_p)$ is integral when $p\equiv 3({\rm mod}\; 4)$, namely it satisfies: 
\[ \widetilde{\mathcal L_p}\subseteq \widetilde{\Lambda_p}=SU(\O_p).\]
Note that $\widetilde{\Lambda_p}=S\mathbb U(\Z)$ is a lattice within the linear algebraic group $S\mathbb U$ defined over $\Q$. There is a similar {\em projective} representation 
$\rho_p:\Gamma_g\to PU$, whose image is  
${\mathcal L_p}={\rho}_{p}({\Gamma_g})\subset \Lambda_p=PU(\mathbb \O_p)$.   

Knowing that the images  $\widetilde{\mathcal L_p}$ of the mapping class groups 
are (generically) infinite (see \cite{F}), Larsen and Wang proved that for prime $p\geq 5$ they are topologically dense in $SU$. Eventually the author showed in \cite{F13} that:  
\begin{proposition}\label{zariskidense}
The image $\widetilde{\mathcal L_p}$ is Zariski dense in 
$S\mathbb U$, if the level $p$ is prime, $p\geq 5$ and $g\geq 2$.  
\end{proposition}

There are several immediate questions which one could ask concerning the 
structure of the group $\mathcal L_p$ and whose answers might shed light on the structure 
of mapping class groups. The following seem to be unknown.    

\begin{question}[Arithmeticity]\label{arithmetic}
Is the group $\mathcal L_p$ arithmetic, namely of 
finite index in the higher rank lattice $\Lambda_p=PU_{g,p}(\O_p)$? 
\end{question}

\begin{question}[Local rigidity]\label{locallyrigid}
Is it true that  quantum representations are locally rigid within $U_{g,p}$ for prime $p\geq 5$, $g\geq 2$? What about their rigidity in $U(N)\supset U_{g,p}$ or $GL(N,\C)$? 
\end{question}

\begin{question}[Injectivity]\label{abstractimage}
If $g\geq 2, p\geq 5$ is prime, then is $\mathcal L_p$ isomorphic to $\Gamma_g(p)$? 
\end{question}

\noindent It is proved in \cite{FuPi1} that a positive answer to Question \ref{abstractimage} implies a negative  answer to Question \ref{arithmetic}. The arithmeticity of various monodromy groups was already intensively  
studied in the literature, starting with the non-arithmetic examples of Nori (\cite{Nori}), the study of thin groups in \cite{Sarnak}, etc.  Venkataramana (see \cite{Venka}) solved the analog of conjecture \ref{arithmetic} in the case of the Burau representation of braid groups $B_{n+1}=\Gamma_{0,1}^{n+1}$ at roots of unity of small order $d\leq \frac{n}{2}$, namely precisely the case where there are unipotents. The unipotent-free case is yet unsolved (see also \cite{Mcmullen}) 
even for the Burau representation.  
Moreover, another construction of mapping class group representations, which will be explained later in this article, was shown to provide finite-dimensional representations with 
arithmetic images in \cite{GLLM}. In this case also the existence of (many) unipotent elements was 
a key ingredient in the proof of arithmeticity.

Note that $\mathcal L_g=\rho_p(\Gamma_g)$ is  also a linear group. Indeed  $\mathcal L_g$ is 
the quotient of the linear group 
$\widetilde{\rho_p}(\widetilde{\Gamma_g})$ by the central finite subgroup 
$\widetilde{\rho_p}(Z(\widetilde{\Gamma_g}))$ where $Z(\widetilde{\Gamma_g})$ denotes 
the center of $\widetilde{\Gamma_g}$. 
Thus $\mathcal L_p$ is residually finite 
and hence Hopfian. 
Therefore the conjecture above is equivalent to the fact that 
$\ker \rho_p= \Gamma_g[p]$.

Recall that Bridson proved that $\Gamma_g$ representations into $GL(g,\C)$ are 
rigid because $\Gamma_g$ has property $FA_g$. 
This question is also related to whether $\Gamma_h$ has  the property $(T,\mathbf F)$, 
as introduced by Lubotzky and Zimmer, where $\mathbf F$ is the family of all finite-dimensional representations. A group $\Gamma$ has  property  $(T,\mathbf F)$ if the trivial representation is isolated in the set of unitary finite-dimensional 
representations. It is known that $SL(2,\Z\left[\frac{1}{p}\right])$ 
has property $(T,\mathbf F)$ but not property $T$, because the congruence subgroup 
conjecture holds true for this group.

We have the following more general strong rigidity question, which is unlikely to have a 
positive answer: 
 
\begin{question}[Strong rigidity]
Any homomorphism $f:\Gamma_g(p)\to G$ to a semi-simple Lie group 
$G$ of Hermitian type factors through  a homomorphism $S\mathbb U_{g,p} \to G$? 
\end{question}

By Simpson's results (see \cite{Sim}) the validity of the local rigidity conjecture 
implies that quantum representations arise as 
factors of variations of Hodge structures (VHS) over $\Q$. 
Note that it is completely unknown whether a similar result holds for $\Gamma_g$, namely: 

\begin{question}
Let $\Gamma_g\to GL(n,\C)$ be a locally rigid representation for $g\geq 3$. 
Then is the associated $\Gamma_g$-invariant 
flat bundle over the Teichm\"uller space $T_g$ 
 the vector bundle associated to a VHS? 
\end{question}

\subsection{Finite quotients through quantum representations}
If $G$ is a group we denote by $\widehat{G}$ its profinite completion. 
It is known that $P\mathbb U_p(\Z)$ has the congruence subgroup property CSP (see \cite{F13}), so that the profinite completion $\widehat{\Lambda_p}$ is isomorphic to $P\mathbb U(\widehat{\Z})$.
Now, the Strong Approximation theorem due to Nori (\cite{Nori}) and Weisfeiler (\cite{Weis})  can be used to obtain 
information about the profinite completion of $\widehat{\mathcal L_p}$ and hence 
$\Gamma_g(p)$. 
First, let us recall the statement due to Nori for 
algebraic groups defined over $\Q$: 
\begin{theorem}[\cite{No}, Thm.5.4]\label{Nori}
Let $G$ be a connected  linear algebraic group $G$ defined over $\Q$  
and $\Lambda\subset G(\Z)$ be a Zariski dense subgroup. 
Assume that $G(\C)$ is simply connected. Then the 
completion of $\Lambda$ with respect to the congruence  
topology induced from $G(\Z)$ is an open subgroup in 
the group $G(\widehat{\Z})$ of points of $G$ 
over the pro-finite completion $\widehat{\Z}$ of $\Z$.  
\end{theorem}
Then, the Zariski density theorem \ref{zariskidense} along  with 
the Strong Approximation Theorem of Nori-Weisfeiler above imply that the image 
of the homomorphism  
\[\hat i: \widehat{\mathcal L_p}\to P\mathbb U_p(\widehat{\Z})\] 
is an open subgroup, where $\hat i$ denotes the map induced by inclusion 
$i:\mathcal L_p\to P\mathbb U_p(\Z)$ at the level of profinite completions.  

Note that the composition 
\[ \Gamma_g(p)\stackrel{\rho_p}{\to}\mathcal L_p \stackrel{i}{\to} \Lambda_p=P\mathbb U_p(\Z)\]
cannot be an isomorphism onto a finite index subgroup of $\Lambda_p$, as $\Gamma_g(p)$ is not a higher rank lattice (see \cite{FuPi1}). 
If $\mathcal L_p$ were a higher rank lattice then it should be an arithmetic subgroup of 
$P\mathbb U_p(\Z)$ and hence of finite index. 
Moreover, in \cite{AF} it was proved that ${\Gamma_g^1}/{\Gamma_g^1[p]}$ has an infinite series 
of normal subgroups of infinite index in each other. 
This shows some evidence that ${\Gamma_g(p)}$ has many more finite quotients than 
the lattice $\Lambda_p=P\mathbb U_p(\widehat{\Z})$. 

\begin{question}[Arithmetic congruence kernel]
Does the homomorphism induced at profinite completions
\[ \widehat{\Gamma_g(p)}\stackrel{\widehat{\rho_p}}{\to}\widehat{\mathcal L_p} \stackrel{\widehat{i}}{\to}\widehat{\Lambda_p}= P\mathbb U_p(\widehat{\Z})\]
have infinite kernel? Is $\hat{i}$ injective? 
\end{question}
If $\hat{i}$ were not injective, then Question \ref{arithmetic} would have a positive answer. 

\begin{remark}
If $\hat{i}$ were an isomorphism and $\mathcal L_p$ nonarithmetic, then 
the inclusion  $i:\widehat{\mathcal L}_p\to P\mathbb U_p(\widehat{\Z})$ 
would provide  additional counter-examples to a conjecture of Grothendieck. 
If $\mathcal L_p$ were locally rigid  and nonarithmetic then one would expect it 
to be super-rigid. 
If $\mathcal L_p$ were not  locally rigid then it would not have  Kazhdan 
property T since the finite-dimensional unitary representations of groups with property T 
are locally rigid (see \cite{Rapinchuk}). Thus $\Gamma_g(p)$ would not have property T. 
\end{remark}

This gives rise to a new profinite completion $\overline{\Gamma}_g$  
of $\Gamma_g$, we will call the {\em q-congruence} completion.
A principal q-congruence subgroup is the preimage of an open subgroup of  
$\prod_{{\rm prime}\; p} \widehat{\Lambda}_p$. These are precisely the 
intersection of kernels of finitely many  epimorphisms onto $P\mathbb U_{p_i}(\Z/q_i\Z)$. 
A finite-index subgroup of $\Gamma_g$ is a q-congruence subgroup if it contains a 
principal q-congruence subgroup. 

As a consequence of the asymptotic faithfulness of the quantum representations
by Andersen (\cite{A1}), and Freedman--Walker--Wang (\cite{FWW}), we derive 
immediately:
\begin{proposition}
The q-congruence topology is separated, namely $\Gamma_g\to \overline{\Gamma_g}$ is injective.  
\end{proposition} 

This q-congruence topology can be further refined by allowing nonprime $p$. 
However, one should note that Proposition \ref{zariskidense} is not anymore true, when 
$p$ is not prime. Much more the representation $\rho_p$ might even be reducible 
in this situation. In this case we need to consider an intermediary group 
which is the Zariski closure $\mathbb L_p$ of $\mathcal L_p$. Then 
$\mathbb L_p$ is a linear algebraic subgroup of $S\mathbb U_p$ defined over $\Q$. 
Moreover, the analog of the Gilmer-Masbaum integrality theorem might not hold when $p$ is not prime. In this case $\widetilde{\Lambda_{g,p}}$ is a lattice in a product of $p$-adic groups.

The above construction  is based on a single explicit family of quantum representations $\rho_p$ indexed by the level $p$. The finite quotients 
constructed this way seem already form a meaningful and rich enough family (see 
\cite{F13,MR}). However, this is just the simplest possible TQFT, commonly associated with the 
SU(2)/SO(3) gauge groups. Stepping to arbitrary simple Lie groups, like the $SU(n)$ family 
might add further finite quotients. However, many of the technical results used above, like the 
Zariski density theorem \ref{zariskidense}  are yet to be developed in order to obtain clean 
statements.

Eventually we should note that the results obtained above for $\Gamma_g$
could in principle be extended to $\Gamma_{g,k}^r$. The case   
of $\Gamma_g^1$ was first considered by Koberda--Santharoubane (\cite{KS}) and further in \cite{FL,EF}. In particular we have linear algebraic groups 
$\mathbb U_{g,p}^1$, playing the role of $\mathbb U_p$ and projective representations 
$\rho_p:\Gamma_g^1\to P\mathbb U_{g,p}^1$. Notice that the group and 
the representation actually depend on the choice of a nonzero color of the marked point, which was 
chosen to be $p-3$ in \cite{FL}.

Consider the Birman exact sequence for $g\geq 2$: 
\[ 1\to \pi_g\to \Gamma_g^1\to \Gamma_g\to 1.\]
The projective representations $\rho_p^1:\Gamma_g^1\to P\mathbb U_{g,p}^1$ 
factors through $\Gamma_g^1(p)$. It is proved in \cite{AF} that there is an analog 
of the Birman exact sequence: 
\[ 1\to \pi_g(p)\to \Gamma_g^1(p)\to \Gamma_g(p)\to 1\]
where $\pi_g(p)=\pi_g/\pi_g[p]$ is the quotient by the normal subgroup 
generated by the $p$-th powers of  classes of the {\em simple} loops on the surface.  
It appears that $\rho_p^1$ descends to a homomorphism $\rho_p^1:\pi_g(p)\to P\mathbb U_{g,p}^1$, 
whose image will be denoted by $\Pi_{g,p}\subset \Lambda_p^1=  P\mathbb U_{g,p}^1(\Z)$. 
The corresponding objects with $\tilde{}$ will be corresponding 
lifts to  $S\mathbb U_{g,p}^1$.

Then the Zarisky density result in proposition \ref{zariskidense} was  extended  in \cite{FL} to 
the punctured case, as follows: 

\begin{proposition}\label{surfzariskidensity}
The image $\rho_p^1(\widetilde{\Pi}_{g,p})$ is Zariski dense in 
$S\mathbb U_{g,p}^1$, if the level $p$ is prime, $p\geq 5$ and $g\geq 2$.  
\end{proposition}

As a consequence, we obtained in \cite{FL} that: 
\begin{proposition}
q-congruence subgroups of $\Gamma_g^1$ are congruence groups. 
\end{proposition}

Note that this method provided infinitely many finite simple quotients of $\pi_g$ 
which are characteristic.

\begin{question}[Zariski density]
Is it true that the  image of $\Gamma_g^r$ in the corresponding unitary group 
$S\mathbb U_g^r$ is Zariski dense for nontrivial colors on the punctures and 
prime levels $p\geq 5$?
\end{question}

\subsection{Handlebody subgroups}\label{section:handlebody}
Let $\mathcal H_g$ denote the mapping class group of the handlebody $H_g$ of genus $g$. 
We also denote by  $\mathcal H_g^r$ the mapping class group of the handlebody $H_g$ 
with $r$ marked points on the boundary. It is well-known that the restriction of 
homeomorphisms to the boundary induces an injective homomorphism 
\[ \mathcal H_g^r \to \Gamma_g^r\]
which permits to identify handlebody mapping class groups to subgroups of the mapping class groups. 

The action of homomorphisms on the free group $\mathbb F_g=\pi_1(H_g)$ yields a 
surjective homomorphism 
\[ \mathcal H_g \to \Out(\mathcal F_g)\]
whose kernel $Tw(H_g)$ is the group of twists of the handlebody.  
The group $Tw(H_g)$ is generated by the Dehn twists along meridians, namely essential 
simple curves on the surface $\Sigma_g$ bounding disks embedded into $H_g$.  
Let now $K_g=\ker(\pi_1(\Sigma_g)\to \pi_1(H_g))$ be the group of meridian curves. 
Then the Birman exact sequence induces a commutative diagram with exact rows and columns:  

\[
\begin{array}{cccccc}
K_g & \to & Tw(H_g^1) & \to & Tw(H_g) \\
\downarrow & & \downarrow & & \downarrow \\
\pi_g & \to & \mathcal H_g^1 & \to & \mathcal H_g\\
 \downarrow & & \downarrow & & \downarrow \\
\mathbb F_g & \to & \Aut(\mathbb F_g) & \to & \Out(\mathbb F_g)\\
\end{array}
\]

An immediate consequence of 
proposition \ref{surfzariskidensity} and (\cite{FL}, Remark 4.4) is the following: 

\begin{proposition}
The image $\tilde\rho_p^1(\mathcal H_g^1)$ is Zariski dense in $S\mathbb U_{g,p}^1$, if the level $p$ is prime, $p\geq 5$ and $g\geq 2$. In particular, $\mathcal H_g^1$ surjects onto infinitely many 
simple nonabelian groups of the form $P\mathbb U_{g,p}^1(\Z/q\Z)$, for large prime $q$. 
Thus $\mathcal H_g^1$ and $\Gamma_g^1$ are residually simple. 
\end{proposition}

We derive that the Frattini subgroup is trivial for $\Gamma_g^1$, complementing the 
result obtained in \cite{MR2} for closed surfaces: 
\begin{corollary}
If $\Phi_f(G)$ denotes the intersection of all finite index maximal subgroups of $G$, then  
$\Phi_f(\Gamma_g^1)=1$, for $g\geq 2$. 
\end{corollary}

By direct computation of the image of two twists along meridians which intersect in two points, as 
in \cite{KS}, we derive that  $\rho_p^1(K_g)$ and hence  $\rho_p^1(Tw(\mathcal H_g^1))$ are infinite and hence they  are topologically dense in $PU^1_{g,p}$. Then the method of \cite{FL} shows that the above result also holds for the twist groups: 

\begin{proposition}
The image  $\rho_p^1(K_g)$ and so $\rho_p^1(Tw(\mathcal H_g^1))$  is Zariski dense in $P\mathbb U_{g,p}^1$, if the level $p$ is prime, $p\geq 5$ and $g\geq 2$. In particular 
$Tw(\mathcal H_g^1))$ surjects onto infinitely many simple nonabelian groups of the form $P\mathbb U_{g,p}^1(\Z/q\Z)$, for large prime $q$.    
\end{proposition}
This is not surprising, as conjecturally, every irreducible subgroup of $\Gamma_{g}^r$ 
which is not virtually abelian should have a large image by $\rho_p^r$. 
However, the case of $\mathcal H_g^1$ is interesting by itself. In fact 
the family of finite quotients obtained above  for large $q$ do not separate the subgroup $\mathcal H_g^1$ within $\Gamma_g^1$, as both groups have the same images under $\rho_p^1$.
Let us project down to $\mathcal H_g$ by means of the forgetful homomorphism 
$p:\Gamma_g^1\to \Gamma_g$. 
Although $\rho_p(\Gamma_g)$ is Zariski dense  in $P\mathbb U_{g,p}$, 
the image $\rho_p(\mathcal H_g)$ of the handlebody group is not. Indeed the restriction 
of $\rho_p$ to $\mathcal H_g$ is not  even irreducible, as it preserves the null-vector 
$w_{g,p}\in W_{g,p}$, which is the vector associated to the handlebody $H_g$. This holds more generally for any subgroup of mapping classes extending 
to a compact 3-manifold. 
In particular we have a map $\theta_p: \Gamma_g/\mathcal H_g\to PW_{g,p}$ defined 
by 
\[ \theta_p(x)= \rho_p(x)w_{g,p}\in PW_{g,p}\]
where $PW$ denotes the projective space associated to the vector space $W$. 
The topological density of $\rho_p(\Gamma_g)$ implies that the image of 
$\theta_p$ is topologically dense in $PW_{g,p}$ and in particular it is infinite, when $p$ is prime. A simpler argument consists in following the proof in \cite{F} for the 
infiniteness of $\rho_p(\Gamma_g)$. There is a subrepresentation of $\rho_p|_{B_3}$ 
which can be identified with the Burau representation at a root of unity. 
But these representations have not finite orbits (see e.g. \cite{FK}).  
Since the map $\theta_p$ factors through $\Gamma_g(p)/\mathcal H_g(p)$, we derive: 

\begin{proposition}
The groups $\mathcal H_g^r(p)$  have infinite index in $\Gamma_g^r(p)$, when the latter are infinite. 
\end{proposition}

This method extends readily to prove the following: 

\begin{proposition}
Let $G\subset \Gamma_{g}$ be a reducible subgroup. Then 
$G(p)$ is of infinite index in $\Gamma_{g}(p)$, for prime $p\geq 5$, $g\geq 2$. 
\end{proposition}
\begin{proof}
It is enough to consider the case when $G$ is the stabilizer of a simple closed 
nonperipheral curve $\gamma$ on $\Sigma_{g,k}^r$. Then $\rho_p(G)$ is centralized 
by $\rho_p(T_{\gamma})$ and hence its topological closure is a proper subgroup of 
$PU_{g,p}$. The map sending $x\in \Gamma_g/G$ to the class of $\rho_p(x)$  
in the homogeneous space obtained by quotienting $PU_{g,p}$ by the closure of $\rho_p(G)$, 
has infinite image, by density. But this map factors through  $\Gamma_g(p)/G(p)$, 
thereby proving the claim. 
\end{proof}

\begin{question}
If $G$ is an irreducible subgroup of $\Gamma_g$ which is neither virtually abelian nor conjugate 
into a subgroup of mapping classes extending to some compact orientable 3-manifold (like $\mathcal H_g$), does it have irreducible or even Zariski dense image in $S\mathbb U_{g,p}$, for large enough primes $p$? 
Moreover, in the case when it is Zariski dense, is $\rho_p(G)$ of finite index in $\rho_p(\Gamma_g)$? A particularly interesting case is $D_3I_{g,2}$, the second 
commutator group of the Johnson kernel $\mathcal I_{g,2}$. 
If $\varphi\in \Gamma_g$ is such that $\rho_p(\varphi)$ fixes the line spanned by the null-vector 
$w_{g,p}$, for infinitely many $p$, does it follow that  $\varphi\in \mathcal H_g$? 
\end{question}

\subsection{Power quotients of surface groups}
It is known (see e.g. discussion in \cite{FLM}) that arithmetic groups have the 
Congruence Subgroup Property if and only if their profinite completion is boundedly generated as a profinite group, namely a finite product of pro-cyclic groups. 
In our case $\mathbb U_p(\Z)$ have CSP, and thus their profinite completions  $\mathbb U_p(\widehat{\Z})$ are boundedly generated as profinite groups, see also 
\cite{CK}. However, this does not implies that $\mathbb U_p(\Z)$ have bounded generation. 
Note that $\Gamma_g$ are not boundedly generated (see \cite{FLM}), as one proved that their pro-$p$ 
completion is not a $p$-adic analytic group. We can slightly improve this, by using a recent result from 
\cite{CRRZ}:

\begin{proposition}
$\pi_g/\pi_g[p]$ is not boundedly generated for large $p$. 
If $\Gamma_g/\Gamma_g[p]$ is infinite, then it is not boundedly generated. 
\end{proposition}
\begin{proof}
A finitely generated infinite torsion group is not boundedly generated. 
As $\pi_g/\pi_g[p]$ surjects onto the Burnside group $\pi_g/\pi_g^p$ associated  to $\pi_g$, which is 
infinite for large $p$, we derive the claim.

According to \cite{CRRZ} every linear group over a field of characteristic zero which is  
anisotropic (i.e. does not contain unipotent elements) and is boundedly generated must be virtually abelian. 
Now $\widetilde{\rho}_p(\widetilde{\Gamma_g})$ is anisotropic, because it is contained in a unitary group. 
Moreover, as soon as it is infinite,  $\widetilde{\rho}_p(\widetilde{\Gamma_g})$ is not virtually solvable, 
since it contains a free non-abelian subgroup. It follows that $\widetilde{\rho}_p(\widetilde{\Gamma_g})$ 
is not boundedly generated. Furthermore, $\Gamma_g/\Gamma_g[p]$ is not boundedly 
generated either, since it  surjects onto the quotient of $\widetilde{\rho}_p(\widetilde{\Gamma_g})$ by 
a finite cyclic central subgroup. 
\end{proof}

\subsection{Infinite index subgroups of $\Gamma_g(p)$}
In \cite{AF} one proved that $\Gamma_g^1(p)$ have infinite 
nested sequences of normal subgroups, each one of infinite index into the previous one. 
This shows that  $\Gamma_g^1(p)$ is far from being a higher rank lattice. 
This last result also holds for $\Gamma_g(p)$, although the former statement is 
unknown. However, we can infer from \cite{BHMS,DHS}  that $\Gamma_g(p)$ is SQ-universal, namely every countable group is a subgroup of some quotient of it. 

By our results above we can see that every finite group is also a finite quotient of 
$\Gamma_g$ (see \cite{F13,MR}). However this question is open for $\Gamma_g(p)$. 
An easy consequence of deep results of Margulis-Soifer is: 

\begin{proposition}
The groups $\Gamma_g(p)$, where 
$p\geq 5$, if $p$ is odd, and $p/4\geq 5$ when $p$ is even, $g\geq 3$ and  additionally 
$(g,p)\neq (2,24)$
admit (uncountably many) free infinite-index maximal subgroups.  
\end{proposition}
\begin{proof}
If $p$ is generic then $\widetilde{\rho_p}(\widetilde{\Gamma_g})$ is a 
linear group which contains free non-abelian groups (see \cite{FK}). 
According to a result of Margulis and Soifer (see \cite{MS})
there exist  uncountably many free (infinitely generated) subgroups 
$\mathbb F_{\infty} \subset \widetilde{\rho_p}(\widetilde{\Gamma_g})$
which are pro-finitely dense, namely they map surjectively onto 
every finite quotient of $\widetilde{\rho_p}(\widetilde{\Gamma_g})$. 
The image $F_{\infty}$ of $\mathbb F_{\infty}$ into the quotient 
$\rho_p({\Gamma_g})$ of $\widetilde{\rho_p}(\widetilde{\Gamma_g})$ by 
a finite central group is still pro-finitely dense. 
We know that $\rho_p$ factors through a homomorphism 
$\overline{\rho_p}:\Gamma_g(p)\to \rho_p(\Gamma_g)$.
Let now $W$ be a  proper maximal subgroup of $\Gamma_g(p)$ 
containing $\overline{\rho_p}^{-1}(F_{\infty})$. 
If $W$ were of finite index in $\Gamma_g(p)$ then 
$\Gamma_g(p)/W$ would be a finite quotient of 
$\rho_p(\Gamma_g)$ in which $F_{\infty}$ maps to the identity. 
This contradicts the fact that $F_{\infty}$ is pro-finitely dense. 
Therefore, $W$ is of infinite index in $\Gamma_g(p)$. 
\end{proof}

\begin{remark}
The subgroups $W$ from above are not normal subgroups of $\Gamma_g(p)$, in general.
If $W$ were normal, then 
the quotient $\Gamma_g(p)/W$ should be an infinite   
simple group, by the maximality of $W$.
\end{remark}

\subsection{A nontrivial homomorphism of $\Gamma_2$ into a rank-1 lattice}\label{sect:fibonacci}
\subsubsection{Hyperelliptic involutions}
The genus 2 closed orientable surface is a double covering of the sphere ramified at 6 points. 
The deck transformation group is generated by a hyperelliptic involution. From \cite{BH}
all mapping classes in $\Gamma_2$  have $\Z/2\Z$-invariant representatives and isotopies 
can be promoted to $\Z/2\Z$-invariant isotopies, so that 
we have the following exact sequence:
\[ 1\to \Z/2\Z\to \Gamma_2\to \Gamma_0^6\to 1\]
where the central kernel is generated by the hyperelliptic involution. 
Further $\Gamma_0^6$ is a quotient of $B_6$.  Specifically, we have the usual presentation:  
\[ B_6=\langle b_1,b_2,...,b_5; b_ib_j=b_jb_i, |i-j|\geq 2, b_ib_{i+1}b_i=b_{i+1}b_ib_{i+1}, i\leq 4\rangle.\]
The usual braid relations are recorded as $\tt Braid$. Then  we have the following  quotient presentations: 
\[ \Gamma_0^6=\langle b_1,b_2,...,b_5; {\tt Braid}, (b_1b_2\cdots b_5)^6=1, b_5b_4\cdots b_2b_1^2b_2\cdots b_4 b_5=1\rangle.\]
We denote by $\Delta_6^2=(b_1b_2\cdots b_5)^6$ the generator of the  infinite cyclic center $Z(B_6)$ of $B_6$. 
The element $h_6= b_5b_4\cdots b_2b_1^2b_2\cdots b_4 b_5$ corresponds to the hyperelliptic involution. 
The spherical braid group $B_6(S^2)$ on $6$ strands on the sphere is then given by 
\[ B_6(S^2)=\langle b_1,b_2,...,b_5; {\tt Braid},  b_5b_4\cdots b_2b_1^2b_2\cdots b_4 b_5=1\rangle.\]
According to Faddell and Neuwirth (\cite{FN}) we have an exact sequence:
\[ 1\to \Z/2\Z\to B_6(S^2)\to \Gamma_0^6\to 1\]
where the kernel $\Z/2\Z$ is central and generated by the image of $\Delta_6^2$ in $B_6(S^2)$. 
On the other hand we have the following presentation of the mapping class group in genus 2 due to Birman and Hilden (see \cite{BH}):
\[ \Gamma_2= \langle b_1,b_2,...,b_5; {\tt Braid}, (b_1b_2b_3)^4=b_5^2, 
[h_6,b_1]=1, h_6^2=1\rangle.\]
We also have the following exact sequence:
\[ 1\to \Z/2\Z\to \Gamma_2\to \Gamma_0^6\to 1\]
where the center $\Z/2\Z$ of $\Gamma_2$ is generated by the image of the hyperelliptic involution $h_6$. 
This exact sequence comes up with another presentation of $\Gamma_2$,  as follows:
\begin{lemma}\label{gen2}
We have 
\[ \Gamma_2= \langle b_1,b_2,...,b_5; {\tt Braid}, (b_1b_2\cdots b_5)^6=1, 
[h_6,b_i]=1, h_6^2=1\rangle.\]
\end{lemma}
\begin{proof}
Remark  that we have 
the following relations in $B_6$:
\[ \Delta_6^2=h_6(b_4b_3b_2b_1)^{5},\]
\[ (b_4b_3b_2b_1)^5=[b_1b_2b_1b_4^{-1}, b_1b_2b_3b_4]( b_1b_2b_3b_4)^5,\]
and 
\[ (b_1b_2b_3b_4)^{5}=(b_1b_2b_3)^4b_4b_3\cdots b_1^2\cdots b_3b_4= 
(b_1b_2b_3)^4b_5^{-2} [b_5,h_6] h_6.\]
Therefore the relation $\Delta_6^2=1$ is equivalent to the 3-chain relation $(b_1b_2b_3)^4=b_5^2$, 
in the presence of the braid relations and the centrality of $h_6$.
\end{proof}
 
The subsurface $\Sigma_{1,2}\subset \Sigma_2$ inherits a hyperelliptic involution which exchanges the 
boundary circles, so that it is a double covering of the sphere ramified at 4 points. Then from (\cite{BH}, see also \cite{FaMa}, 9.4.1) we have an identification between $\Gamma_{1,2}$ and $B_4/Z(B_4)$. We derive the presentation: 

\[ \Gamma_{1,2}=\langle b_1,b_2,b_3; {\tt Braid}, (b_1b_2b_3)^4=1 \rangle.\]

\subsubsection{The Jones representation of $\Gamma_2$}
The Jones representation $J_q:B_g\to GL(5,\Z[q,q^{-1}])$ is the representation of the Hecke algebra at 
$q$ corresponding to the rectangular Young diagram  associated to the partition $2^3$. Specifically we have: 
\[ J_q(b_1)=\left(
\begin{array}{ccccc}
-1 & 0 & 0 & 0 & q \\
0 &- 1 & 1 & 0 & 0 \\
0 & 0 & q & 0 & 0 \\
0 & 0 & 1& -1 & 0 \\
0 & 0 & 0 & 0 & q\\
\end{array}
\right), \; 
J_q(b_2)=\left(
\begin{array}{ccccc}
q & 0 & 0 & 0 & 0 \\
0 & q & 0 & 0 & 0 \\
0 & q & -1& 0 & 0 \\
1 & 0 & 0 & -1 & 0 \\
1 & 0 & 0 & 0 & -1\\
\end{array}
\right), \;
J_q(b_3)=\left(
\begin{array}{ccccc}
-1 & 0 & 0 & q & 0 \\
0 & -1 & 1 & 0 & 0 \\
0 & 0 & q & 0 & 0 \\
0 & 0 & 0& q  & 0 \\
0 & 0 & 1 & 0 & -1\\
\end{array}
\right)
\] 
\[
J_q(b_4)=\left(
\begin{array}{ccccc}
q & 0 & 0 & 0 & 0 \\
1 & -1 & 0 & 0 & 0 \\
0 & 0 & -1 & 0 & q \\
1 & 0 & 0& -1 & 0 \\
0 & 0 & 0 & 0 & q\\
\end{array}
\right), \;
J_q(b_5)=\left(
\begin{array}{ccccc}
-1 & q & 0 & 0 & 0 \\
0 & q & 0 & 0 & 0 \\
0 & 0 & q & 0 & 0 \\
0 & 0 & 1& -1 & 0 \\
0 & 0 & 1 & 0 & -1\\
\end{array}
\right), \;
J_q(\delta_6)=q^2\left(
\begin{array}{ccccc}
0 & 0 & 1 & 0 & q \\
0 & 0 & 0 & 0 & 1 \\
1 & 0 & 0 & 0 & 0 \\
0 & 1 & 0 & 0 & 0 \\
0 & 0 & 0 & 1 & 0\\
\end{array}
\right)
\]
Here $\delta_6=b_1b_2b_3b_4b_5$. 
Observe that this is slightly modified with respect to the original one from (\cite{Jones}, section 10, p. 362), 
such that the eigenvalues of $J_q(b_i)$ are $1$ and $q$, with multiplicities $3$ and $2$ respectively, for reasons which will be become clear later.  

The Jones representation is of {\em Hecke type}, namely it factors through the Hecke algebra $H(q,6)$ 
defined as the quotient algebra:
\[ H(q,n)=\C[B_n]/(b_i^2+(1-q)b_i-q).\] 

\subsubsection{Fibonacci representations in genus 2}
The  Fibonacci TQFT is the $SO(3)$-TQFT at $p=5$. In order to fix completely the theory we have to choose a primitive $10$-th root of unity $A$. 
There are only two colors $\{0,2\}$ and thus we can compute explicitly the dimension of the 
space of conformal blocks $W_g(2^k)$ of the  surface  of genus $g$ with $k$ boundary components 
labeled by the color 2. Recall that the boundary components labeled with the color $0$ could be 
filled in by disks. Then 
\[ \dim W_{g,5}(2^k)=5^{\frac{g-1}{2}}\left (\left(\frac{1+\sqrt{5}}{2}\right)^{g+n-1}+(-1)^g \left(\frac{1-\sqrt{5}}{2}\right)^{g+n-1}\right).\]

If we wish to specify the number of boundary components labeled by $2$ we will write 
$\tilde\rho_{g,5;k}$ for the corresponding  representation in genus $g$ and level $5$.  
\begin{lemma}\label{irred}
The representation $\rho_{g,5;k}$ is irreducible, as soon as $\dim W_{g,5}(2^k)\geq 1$.  
\end{lemma}
\begin{proof}
The irreducibility of representations arising in the $SU(2)$-TQFT for all 
roots of unity of order $4p$, with prime $p$ was proved by Roberts (\cite{Roberts}). The proof works 
ad-literam for the  $SO(3)$-TQFT at roots of unity of order $2p$, with prime $p$.
A different proof is provided in (\cite{FLW}, Prop. 6.4) at $p=5$.   
\end{proof}

Therefore the representation $\tilde\rho_{2,5}$ is irreducible. 
Composition with the obvious surjection $B_6\to \Gamma_2$ provides a projective 
representation still  denoted $\rho_2:B_6\to PU(W_2, H_{A})$, where $H_A$ is the 
Hermitian form associated to the primitive root of unity $A$. 
By the formula above we have $\dim W_2=5$. 

The  first main result of this section is:
\begin{proposition}
The representation $\rho_{2,5}:B_6\to PU(W_2, H_{A})$ is equivalent to the projectivisation of the 
Jones representation $-J_q$ at a  primitive $10$-th root of unity $q=-A^8$. 
\end{proposition}
\begin{proof}
It is proved in (\cite{FLW}, Thm. 6.2) that not only $\rho_{2,5}$ is irreducible but its image is topologically dense 
within $PU(W_2,H_A)$, when $A=\exp\left(\frac{6\pi i}{10}\right)$ (see also \cite{LW} for the more general case). 
Notice that in this case $H_A$ is 
positive definite and the group $PU(W_2,H_A)$ can be identified with the compact group $PU(5)$.
This implies that the image of $\rho_{2,5}$ is Zariski dense in $PU(W_2, H_A)$ for every $10$-th 
primitive root of unity $A$.  

The first observation is that the Zariski density of  the image of $\rho_2$ implies that  
the image of the hyperelliptic involution is trivial, so that $\rho_2$ factors through $\Gamma_0^6$.  

Recall now that we have a linear representation $\widetilde\rho_{2,5}:\widetilde\Gamma_2\to U(W_2,H_A)$ 
of a central extension of $\Gamma_2$ by $\Z$ which lifts $\rho_{2,5}$. The description of $\widetilde\Gamma_2$ in terms of group presentations was given by Gervais in \cite{Gervais}: we only have to replace the chain relation 
by $(b_1b_2b_3)^4=z^{12}b_5^2$, where $z$ is a central infinite-order additional generator. 
According to Lemma \ref{gen2} this amounts to the following presentation: 
\[ \widetilde\Gamma_2= \langle b_1,b_2,...,b_5, z; {\tt Braid}, (b_1b_2\cdots b_5)^6=z^{12}, [z,b_i]=1, 
[h_6,b_i]=1, h_6^2=1\rangle.\]

The pull-back of the central extension 
$\widetilde\Gamma_2\to \Gamma_2$ by the homomorphism $B_6\to \Gamma_2$ 
is a central extension of $B_6$ by $\Z$. 
Arnold (\cite{Arnold}, see also \cite{Ver}, Thm. 4.3) proved that the cohomology of braid 
groups stabilizes $H^k(B_n;\Z)=H^k(B_{2k-2};\Z)$ for $n\geq 2k-2$, so that 
$H^2(B_n;\Z)=H^2(B_2;\Z)=0$, for $n\geq 2$. 
This proves that linear  representations of central extensions of $B_n$ by $\Z$ 
lift to linear representations of $B_n$. Actually the presentation we gave for 
$\widetilde\Gamma_2$ makes it clear that the tautological map on generators 
is a well-defined homomorphism $B_6\to \widetilde\Gamma_2$.
In particular there is a lift $\hat\rho_{2,5}:B_6\to U(W_2, H_{A})$ of $\widetilde\rho_2$. 
Lemma \ref{irred} shows that this linear representation is irreducible. 
Moreover this representation verifies $\hat\rho_{2,5}(h_6)=1$.

The classification of 5-dimensional irreducible representations of $B_6$ was given by Formanek (\cite{For}, see also 
\cite{For-Vazi} for a systematic description). Following (\cite{For-Vazi}, Thm. 14) they are 
of Hecke type, namely they factor through the Hecke algebra $H(q,6)$ for some $q$. Moreover, up to tensoring with 
a 1-dimensional representation these are equivalent to  the specialization of 
either the Burau representation which corresponds to the 
Young diagram associated to the partition $21^4$ or else to another representation which corresponds 
to the Young diagram associated to the partition $2^3$. 

In order to decide which one appears one has to compare the eigenvalues of the Dehn twists 
along with their multiplicities. In the case of $\hat\rho_2(b_i)$ they are 
$(1,1,1,A^8,A^8)$, up to a scalar, for the Burau representation the eigenvalues are  
 $(-1,-1,-1,-1,q)$, while in the case of the Jones representation $J_q$ they are 
$(-1,-1,-1,q,q)$. 
 It follows that $q=-A^8$ is a primitive 10-th root of unity and that $\hat\rho_2$ cannot be 
 equivalent to the Burau representation. We derive that $\hat\rho_2$ is equivalent to 
 $-J_{-A^8}$. Notice also that the Jones representation factors through $\Gamma_{0,6}$ and Dehn twists 
 are of order 5.
\end{proof}

The non-degenerate Hermitian form $H_{A}$ has signature $(5,0)$, when $A=\exp\left(\pm \frac{6\pi i}{10}\right)$ 
and signature $(1,4)$, when $A=\exp\left(\pm \frac{2\pi i}{10}\right)$, respectively. 
In particular we obtain a homomorphism $f:\Gamma_2(5)\to PU(1,4)$.  

In \cite{AF} the authors proved that $\mathcal K_g(p)=I_{g,2}(p)$ is a K\"ahler group, for any $g\geq 2$ and odd $p$. In particular, $\mathcal K_2(p)=\mathcal T_2(p)$ is a K\"ahler group. 
Instead of restricting $f$ to $\mathcal T_2(5)$ we will work directly with 
$\Gamma_2(5)$, knowing that in \cite{EF} it was proved that $\Gamma_g(p)$ is a K\"ahler group.
Let $X_2(5)$ be a complex projective variety with fundamental group $\Gamma_2(5)$.
The constructions in \cite{AF,EF} show that we can take for $X_2(5)$ a compactified 
moduli space of curves with level structure, and in particular $\dim_\C X_2(5)=3$.  
Consider next a finite index torsion-free subgroup $\Lambda\subset SU(1,4)(\mathcal O_{10})$ and 
let $J$ be its preimage within $\Gamma_2(5)$.  
 
According to Eells--Sampson $f:J\to \Lambda$ is induced by a harmonic map 
$F:X_2(5)\to Z$, where $Z$ is the compact complex hyperbolic manifold 
$H_\C^4/\Lambda$. Moreover, Carlson--Toledo proved in (\cite{CT}, Thm. 7.2) that 
either $F$ has rank at most 2, or else we can take $F$ to be holomorphic or anti-holomorphic.  
It seems that  $F$ can be taken to be holomorphic or anti-holomorphic, some evidence 
being provided by the following: 

\begin{proposition}
The virtual cohomological dimension of $f(\Gamma_2(5))$ is at least $4$.  
\end{proposition}
\begin{proof}
Consider the stabilizer of a nontrivial nonseparating simple closed curve $\gamma$ on $\Sigma_2$. 
This is isomorphic to $\Gamma_{1,2}/\langle ab^{-1}\rangle$, where $a,b$ denote 
the Dehn twists along the boundary components. 
If we label $\gamma$ by the color $2$ then we obtain a subspace $V$ of the space of conformal blocks in genus 2, which is invariant by the action of the stabilizer.  Its orthogonal $V^{\perp}$ with respect to the 
Hermitian form $H_A$ is the 2-dimensional subspace corresponding to the label $0$ of $\gamma$.
Since $\Gamma_{1,2}$ is isomorphic to $B_4/Z(B_4)$ we obtain two 
representations of $\beta:B_4\to GL(V)$, and 
$\gamma:B_4\to GL(V^{\perp})$. 

Now, both $\beta$ and $\gamma$ are factors  of the restriction 
of $\hat\rho_{2,5}$ to $B_4$.  The restriction of a Hecke type representation of $B_n$ to the  
subgroup $B_{n-1}\subset B_n$ splits into irreducible components indexed 
by the Young subdiagrams with one box less. It follows that 
the restriction of $J_{-A^8}$ to $B_5$ is the irreducible representation with Young diagram $2^21$ and the 
restriction to $B_4$ has two irreducible components corresponding to $21^3$ (of dimension 3) and
 $2^2$ (of dimension 2). Notice that indeed these two representations 
 are irreducible of the same dimension as the ones for generic 
 $q$ (see \cite{For-Vazi}).  It follows that $\beta$ is equivalent to the Burau representation $\beta_{-A^{8}}$ while 
 $\gamma$ is equivalent to the Hecke type representation associated to the partition $2^2$.

Recall also that the curves labeled $0$ in conformal blocks can be filled in by disks. This means that the 
projectivization of $\gamma$  factors through the mapping class group of the torus. Thus 
$\gamma$ factors through  the  folding homomorphism $i:B_4\to B_3$ given by 
$i(b_1)=b_1, i(b_2)=b_2, i(b_3)=b_1$. It is known that $\ker i\subset B_4$ is the free 
group on two generators $b_1b_3^{-1}, b_2b_1b_3^{-1}b_2^{-1}$. 
Therefore the $B_3$ representation  obtained from $\gamma$ is the restriction 
of $-J_{-A^8}$ to $B_3$, namely the Burau representation $\beta_{q}$ of $B_3$ at the primitive
 $10$-th root of unity $q=-A^{8}$. 
 
Recall that the reduced Burau representation $\beta_q$ is given by: 
\[\beta(b_1)= \left (\begin{array}{ccc}
q &- 1 & 0 \\
0 & -1 & 0 \\
0 & 0 & 1\\
\end{array}\right), \beta(b_2)= \left (\begin{array}{ccc}
1 & 0 & 0 \\
q & q &- 1 \\
0 & 0 & -1\\
\end{array}\right), \beta(b_3)= \left (\begin{array}{ccc}
1 & 0 & 0 \\
0 & -1 & 0 \\
0 & -q & q\\
\end{array}\right).\] 

Following (\cite{FK}, Prop.3.1) the image of $\gamma(B_4)=\beta_{q}\circ i(B_4)
\subset U(V^{\perp}, H_A)$ is $B_3/B_3[5]$, where 
$B_3[5]$ is the normal subgroup of $B_3$ generated by $b_i^5$. 
Moreover this group is isomorphic to $GL(2,\Z/5\Z)$, of order 600.

Further,  $\beta$  is equivalent to the Burau representation $\beta_{q}$ of $B_4$ 
at $q$. 
Moreover this representation preserves the Hermitian form $H_{q}$, whose signature is 
$(3,0)$ when $q=\exp\left(\frac{\pm 6\pi i}{10}\right)$  (i.e. $A=\exp\left(\frac{\pm 2\pi i}{10}\right)$) 
and of signature $(1,2)$ when $q=\exp\left(\frac{\pm 2\pi i}{10}\right)$ 
(i.e.  $A=\exp\left(\frac{\pm 6\pi i}{10}\right)$ ). The real points of the linear algebraic group 
obtained by the restriction of scalars $\Q(q+\overline{q}):\Q$   
is therefore isomorphic to the product $U(3)\times U(1,2)$. 
Therefore $\beta_{q}(B_4)\subset U(V)$ is a discrete subgroup. 

McMullen proved in \cite{Mcmullen} that in this case the image of the Burau representation of $B_4$ at  $q=\exp\left(\frac{6\pi i}{10}\right)$ is a lattice in $PU(1,2)=PU(V, H_{q})$. 
This is a cocompact arithmetic lattice. 
In particular the image of $f(\Gamma_2(5))$ contains a cocompact lattice in  $PU(1,2)$ 
whose virtual cohomological dimension is $4$. Now, the virtual cohomological 
dimensions decreases when passing to subgroups (see \cite{Brown}, VIII, 11, ex.1, Prop. 2.4)  and hence the claim follows.  
\end{proof}

\section{Local rigidity of Weil representations}\label{sect:Weil}
The simplest test for the rigidity questions is the case when the representations have finite images. Among those, the Weil representations were intensively studied (see \cite{FuPi3} for a brief history). Weil representations could be defined also by geometric quantization or 
in the quantum groups framework. They were rediscovered within the framework of Chern--Simons theory with abelian gauge group $U(1)$. 
  
Let $k \geq 2$ be an integer, and denote by  
$\langle , \rangle$ the standard bilinear form on 
$(\Z/k\Z)^g \times (\Z/k\Z)^g\to  \Z/k\Z$. The Weil
 representation we consider is a representation in the 
unitary group of the 
complex vector space $\C^{(\Z/k\Z)^g}$ endowed with its 
standard Hermitian 
form. Notice that the canonical basis of this vector space 
is canonically labeled by elements in $\Z/k\Z$.

It is well-known that $Sp(2g,\Z)$ is generated by the matrices having one
of the following forms:
$
\left ( \begin{array}{cc}
    \mathbf 1_g & B \\
    0 & \mathbf 1_g
\end{array}  \right)$
where $B=B^{\top}$ has integer entries,
$
\left ( \begin{array}{cc}
    A & 0       \\
    0 & (A^{\top})^{-1}
\end{array}  \right)$
where $A\in GL(g,\Z)$ and  
$
\left ( \begin{array}{cc}
    0 & -\mathbf 1_g \\
    \mathbf 1_g & 0
\end{array}  \right)$.

We can now define the Weil representations 
on these generating matrices as follows:
\begin{equation}
\rho_{g,k}
\left ( \begin{array}{cc}
     \mathbf 1_g & B\\
     0 & \mathbf 1_g
\end{array}  \right)
= {\rm diag}\left(\exp\left(\frac{\pi \sqrt{-1}}{k}\langle m,Bm\rangle\right)\right)_{ m\in(\Z/k\Z)^{g}},
\end{equation}
where ${\rm diag}$ stands for diagonal matrix with given entries;

\begin{equation}
\rho_{g,k}
\left ( \begin{array}{cc}
   A & 0\\
   0 & (A^{\top})^{-1}
\end{array}  \right)
= (\delta_{A^{\top}m,n})_{m,n\in(\Z/k\Z)^{g}},
\end{equation}
where $\delta$ stands for the Kronecker symbol;
\begin{equation}
\rho_{g,k}
\left ( \begin{array}{cc}
    0 & -\mathbf 1_g\\
   \mathbf  1_g & 0
\end{array}  \right)
=k^{-g/2}\exp\left(-\frac{2\pi\sqrt{-1}\langle m,n\rangle}{k}\right)_{m,n\in({\Z}/k{\Z})^{g}}.
\end{equation}

For even $k$ these formulas define a unitary projective  
representation $\rho_{g,k}$ of $Sp(2g,\Z)$ in 
$U(k^g)/R_8$, where $R_8\subset U(1)\subset U(\C^N)$ is the subgroup of scalar 
matrices whose entries are roots of unity of order $8$.  
For odd $k$ the same formulas define representations of 
the theta subgroup $Sp(2g, 1,2)$. There is however an extension of this representation 
to the whole symplectic group $Sp(2g,\Z)$, as defined by Murakami, Ohtsuki and Okada in 
\cite{MOO}. 
Notice that by construction $\rho_{g,k}$ factors through 
$Sp(2g, \Z/2k\Z)$ for even 
$k$ and through the image of the theta subgroup in $Sp(2g,\Z/k\Z)$ 
for odd $k$.

In \cite{FuPi3} we proved that 
the projective Weil  representation $\rho_{g,k}$ of 
$Sp(2g,\Z)$, for $g\geq 3$ and  even $k$ does not lift to 
linear representations of $Sp(2g,\Z)$, namely it determines  
a generator of $H^2(Sp(2g,\Z/2k\Z); \Z/2\Z)$ and hence a (universal) central extension 
$\tilde{Sp}(2g,\Z/2k\Z)$ of $Sp(2g,\Z/2k\Z)$ by $\Z/2\Z$.  In fact 
$H_2(Sp(2g,\Z/k\Z))=\Z/2\Z$, if and only if $k$ is divisible by $4$, while 
for other cases, it vanishes (see \cite{FuPi2}). 
For odd $k$ it was already known that Weil representations 
did not detect any non-trivial element, i.e. that the projective 
representation $\rho_{g,k}$  lifts to a linear representation.
By pulling-back the central extension of $\tilde{Sp}(2g,\Z/2k\Z)$  to 
$\Gamma_g$ we obtain a central extension $\tilde{\Gamma}_g$ by $\Z/2\Z$, endowed 
with a linear representation $\rho_{g,k}^{U(1)}$ into $U(k^g)$. 
We then have an exact sequence 
\[ 1\to \Gamma_g((k))\to \tilde{\Gamma}_g\to \tilde{Sp}(2g,\Z/k\Z)\to 1\]
where 
\[  \Gamma_g((2k))=\ker(\Gamma_g\to Sp(2g,\Z/2k\Z))\]
is the so-called abelian level $2k$ mapping class group. 

\begin{proposition}
If $g\geq 3$ the  $U(1)$ representations 
$\rho_{g,k}^{U(1)}: \tilde\Gamma_g\to U(k^g)\subset GL(n,\C)$ are locally rigid as 
$GL(n,\C)$ representations. 
\end{proposition}
\begin{proof} 
We have the five term exact sequence in cohomology: 
\[ 0 \to H^1(\tilde{Sp}(2g,\Z/2k\Z), \frak{gl}_n^{\Gamma_g((2k))}) 
\to H^1(\tilde\Gamma_g, \frak{gl}_n)\to  H^1(\Gamma_g((2k)),\frak{gl}_n)^{\tilde{Sp}(2g,\Z/2k\Z)}\to\] 
\[\to H^2(\tilde{Sp}(2g,\Z/k\Z), \frak{gl}_n^{\Gamma_g((2k))})\to H^2(\tilde\Gamma_g, \frak{gl}_n).\]  
Here we use the cohomology with twisted coefficients, where 
the action of $\tilde\Gamma_g$ on the Lie algebra $\frak{gl}_n$ 
is by $Ad\circ \rho_{g,k}^{(U(1)}$. Since the action of 
 $\Gamma_g((2k))$ is trivial we have 
\[ H^1(\Gamma_g((2k)),\frak{gl}_n)= {\rm Hom}(\Gamma_g((2k)),\frak{gl}_n)= 0, {\rm when }\; g\geq 3. \]
In fact $\frak{gl}_n$ is considered here with its structure 
of abelian group and thus  any homomorphism 
$\Gamma_g((2k))\to \frak{gl}_n$ factors through $H_1(\Gamma_g((2k)))$. 
Now, McCarthy proved in \cite{McC} that $H^1(\Gamma_g((k)))=0$, for every 
$k$ and $g\geq 3$.  
This is actually true for any finite index subgroup of $\Gamma_g$ 
which contains the Torelli group. See for instance \cite{Sato} for the 
precise description of the finite group $H_1(\Gamma_g((k)))$. 

Then we have: 
\begin{eqnarray*} H^1(\tilde{Sp}(2g,\Z/k\Z), \frak{gl}_n^{\Gamma_g((2k)))})& =&
 H^1(\tilde{Sp}(2g,\Z/2k\Z), \frak{gl}_n)=\\ 
& =& H^1(\tilde{Sp}(2g,\Z/2k\Z), \frak{gl}_{k^g}) \oplus 
H^1(\tilde{Sp}(2g,\Z/2k\Z), \frak{gl}_{n-k^g}). 
\end{eqnarray*}

Now $H^1(\tilde{Sp}(2g,\Z/2k\Z), \frak{gl}_{n-k^g})=0$ by the universal coefficients theorem, since the action is trivial and 
 $H^1(\tilde{Sp}(2g,\Z/2k\Z))=0$, when $g\geq 3$, because this is a universal central extension. 
Eventually, the five-term exact sequence above implies that: 
\[ H^1(\tilde\Gamma_g, \frak{gl}_n)=0\]
so that $\rho_{g,k}^{U(1)}$ is locally rigid in $GL(n,\C)$, following 
Weil's criterion. 
\end{proof}

\begin{proposition}
The representation $\rho_{2,k}^{U(1)}:\tilde\Gamma_2\to U(k^2)\subset GL(n,\C)$ at genus 
$g=2$ is not locally rigid, if $k\geq 4$ is even or divisible by 3  
and $n >k^2$. 
\end{proposition}
\begin{proof}
Notice that $H^1(\Gamma_2((k)))$ is non-trivial 
when $k$ is divisible by $2$ or $3$ (see \cite{McC}).
In particular it contains a factor $\Z$.

From (\cite{LuMa}, p.37-38 and more generally Prop. 10.1 from 
\cite{Brown}) we have 
$H^1(F,\frak{gl}_m)=0$ for any representation 
of  a finite group $F$ in the $GL(m,\C)$. 
In fact finite groups are reductive and hence they are rigid.   
In particular, we have   $H^1(\tilde{Sp}(4,\Z/2k\Z), \frak{gl}_n)=0$. 

\begin{lemma}\label{vanishing} 
\[ H^2(\tilde{Sp}(2g,\Z/2k\Z), \frak{gl}_n)=0, \; g\geq 2, k\neq 2.\]
\end{lemma}
\begin{proof}
This follows from the 
following classical fact (see Prop. 2.1 of \cite{Brown}): 
If $G$ is a finite group 
and $M$ is a $G$-module which is also a $\mathbb K$-vector space
for a field $\mathbb K$ whose characteristic 
does not divide the order of $G$ then 
$H^j(G,M)=0$, when $j >0$. In particular this is true in characteristic zero. 
\end{proof}

The five term exact sequence above shows that 
\[ H^1(\tilde\Gamma_2, \frak{gl}_n)\cong H^1(\tilde\Gamma_2((k)),\frak{gl}_n))^{\tilde{Sp}(4,\Z/2k\Z)}\]
But now 
\[H^1(\tilde\Gamma_2((2k)),\frak{gl}_n)^{\tilde{Sp}(4,\Z/2k\Z)}\cong 
{\rm Hom}(\Gamma_2((2k)),\frak{gl}_n)^{\tilde{Sp}(4,\Z/2k\Z)} \supset \frak{gl}_n^{\tilde{Sp}(4,\Z/2k\Z)}
\]
since $H_1(\Gamma_2((2k)))\supset \Z$. 
In particular if $n > k^2$ then the unitary representation
of  $\tilde{Sp}(4,\Z/2k\Z)$ in $\frak{gl}_n$  keeps invariant 
the orthogonal of $\frak{gl}_{k^2}$ within $\frak{gl}_n$. 
This implies that $H^1(\tilde\Gamma_2, \frak{gl}_n)\neq 0$, so that the representations 
$\rho_{2,k}^{U(1)}$ are not locally rigid. 
\end{proof}

\begin{remark}
Since $Sp(2g,\Z)$, $g\geq 2$ have property $F$, their unitary representations 
have finite images and thus they are discrete. In particular any small 
deformation of the representation $\rho_{g,k}$ is still a discrete representation 
in $U(k^g)$. Therefore Selberg's proof from \cite{Sel} 
can be  used to show that the images are isomorphic.
Since $Sp(2g,\Z)$ are linear reductive (see \cite{LuMa}) 
all its linear representations, in particular the $U(1)$ representations 
$\rho_{g,k}: Sp(2g,\Z)\to U(k^g)\subset GL(n,\C)$, are locally rigid as 
representations in $GL(n,\C)$, when $g\geq 2$ and $n\geq k^g$.
The linear reductivity  is a consequence of the Margulis super-rigidity.
In the unitary case this also follows from the easier fact 
that $Sp(2g,\Z)$ has property T. 
\end{remark}

\begin{remark}
It seems that the $SU(2)/SO(3)$ quantum representations $\rho_{p}$ having 
finite image are locally rigid, if $g\geq 3$.  
It suffices to show that 
\[ H^1(\ker \rho_{p},\frak{gl}_n)=0.\]
At $4$-th roots of unity (and hence $p=8$) this could follow from the 
description due to Masbaum and Wright (see \cite{Ma1,Wr2})
of $\ker \rho_{8}$, and the fact that finite-index 
subgroups of $\Gamma_g$ containing the Johnson kernel $\mathcal K_g$ have finite abelianization, 
according to Putman (see \cite{Pu}).  
For $p=12$ this might use the results from \cite{Wr1}. 
\end{remark}

\section{Tangent representations from moduli spaces}
\subsection{Mapping class groups as  (outer) automorphisms groups}
Now, let $\Gamma_{g,k}^{r,1}\subset \Gamma_{g,k}^{r+1}$ denote the index $r+1$ subgroup of mapping classes 
of those  homeomorphisms which fix one marked point. Then we have the more general statement: 
 \[ \Gamma_{g,k}^{r |1} = {\rm Aut}^+(\pi_{g,k}^r; P_r, \mathbf P^{\partial}_k).\]
Here $P_{r}$ consists of the $r$ conjugacy classes of peripheral loops with the exception of the one around the marked basepoint. 
 
Then  we have the following commutative diagram consisting of two exact sequences corresponding to 
Birman's exact sequence,  connected by  isomorphisms provided by the Dehn--Nielsen--Baer theorem: 
\[\begin{array}{ccccccccc}
1&\to& \pi_{g}^r/Z(\pi_{g}^r)& \to &P\Gamma_{g}^{r+1} &\to &P\Gamma_{g}^r &\to& 1\\
& & \downarrow &  & \downarrow & & \downarrow & & \\
 1 &\to & \pi_{g}^r/Z(\pi_{g}^r) & \to & {\rm Aut}^+(\pi_{g}^r; \mathbf P_{r}) & \to &  {\rm Out}^+(\pi_{g,k}^r;\mathbf P_{r}) & \to &  1\\
 \end{array}
  \]
  We have also a similar commutative diagram in the non pure case: 
  \[\begin{array}{ccccccccc}
1&\to& \pi_{g}^r/Z(\pi_{g}^r)& \to &\Gamma_{g}^{r |1} &\to &\Gamma_{g}^r &\to& 1\\
& & \downarrow &  & \downarrow & & \downarrow & & \\
 1 &\to & \pi_{g}^r/Z(\pi_{g}^r) & \to & {\rm Aut}^+(\pi_{g,k}^r;  P_r) & \to &  {\rm Out}^+(\pi_{g,k}^r,P_r) & \to &  1\\
 \end{array}
  \]
Here $Z(G)$ denotes the center of the group $G$. 
Now, the group $\pi_{g}^r$ is either a free group of rank $2g+r-1$, if $r>0$ or else a surface group.  In particular it is centerless when $2g+r-2>0$.

Consider a surface with one boundary component $\Sigma_{g,1}^r$ and take the basepoint to be on the boundary component.  
  Therefore, the basepoint is automatically invariant by the pure mapping class group.
  It follows that we also have the alternative description: 
\[ \Gamma_{g,1}^{r} = {\rm Aut}^+(\pi_{g,1}^{r}; [\partial \Sigma_{g,1}^r],  P_r). \]
Notice that homeomorphisms of $\Sigma_{g,1}^r$  automatically preserve  the orientation. 
Denote by 
\[ \tau: \Gamma_{g,1}^{r} \to {\rm Aut}^+(\pi_{g,1}^{r}; [\partial \Sigma_{g,1}^r],  P_r) \]
the natural isomorphism, which generalizes the usual Artin representation. 
The following is rather well-known: 
\begin{lemma}\label{split}
There is an isomorphism between $\Gamma_{g,1}^{r |1}$ and the semi-direct product 
$\pi_{g,1}^r\rtimes_{\tau} \Gamma_{g,1}^{r}$, if  and $2g+r-1>0$, which restricts to an isomorphism between 
the pure mapping class group $P\Gamma_{g,1}^{r+1}$ and the  semi-direct product 
$\pi_{g,1}^r\rtimes_{\tau} P\Gamma_{g,1}^{r}$. 
\end{lemma}
\begin{proof}
The embedding of $\Sigma_{g,1}^r$ into $\Sigma_{g,1}^{r+1}$ as the complement of a punctured annulus $\Sigma_{0,2}^1$ induces 
injective homomorphisms $\pi_1(\Sigma_{g,1}^r,*) \to \pi_1(\Sigma_{g,1}^{r+1},*) $ 
and $\Gamma_{g,1}^r\to \Gamma_{g,1}^{r |1}$. Here $*$ is a basepoint on the boundary component of the punctured annulus. 
This provides a splitting of the Birman exact sequence above. Moreover,  the 
action of the subgroup $\Gamma_{g,1}^{r}$ on the subgroup $\pi_1(\Sigma_{g,1}^r,*)$ coincides with $\tau$. Therefore 
$\Gamma_{g,1}^{r |1}$ is isomorphic to the given semi-direct product. 
\end{proof}

It is easy to see that there is a more general version, in which we consider mapping class groups instead of pure ones 
(see e.g. \cite{DF}). The corresponding semi-direct product is now isomorphic to the stabilizer of the last puncture in the 
mapping class group of the surface with one extra puncture, provided the surface has boundary.

\subsection{Geometric actions of  (outer) automorphisms groups on moduli spaces}
Let $\pi$ be a finitely generated group and $G$ a connected Lie group. We denote by 
${\rm Hom}(\pi,G)$ the space of representations of $\pi$. 
The group ${\rm Aut}(\pi)$  acts on ${\rm Hom}(\pi,G)$ by right composition:  
 \[ (\varphi \cdot \rho)(x)=\rho(\varphi^{-1}(x)), \; {\rm for}\; \varphi\in {\rm Aut}(\pi) , \rho\in {\rm Hom}(\pi,G), x\in \pi.\]
 This is a real algebraic action. 
Let now $\Mod_{\pi,G}$ be the 
character variety of representations $\pi\to G$, or the GIT quotient ${\rm Hom}(\pi,G)/G$. 
Then the action 
\[ {\rm Aut}(\pi) \times {\rm Hom}(\pi,G) \to {\rm Hom}(\pi,G)\]
above passes to a quotient action of 
\[{\rm Out}(\pi) \times  \Mod_{\pi,G} \to \Mod_{\pi,G}.\]

Let $F$ be a finitely generated group. Fix a surjective homomorphism $\rho:\pi \to F$ whose kernel $\ker\rho$ is denoted by $K$ 
and consider its stabilizer, i.e. the  subgroup of those elements whose induced action on $F$ via $\rho$ is trivial:
\[ {\rm Aut}(\pi,\rho)=\{\varphi; \rho(\varphi(x))=\rho(x), {\rm for \; any }\; x\in \pi\}\subset {\rm Aut}(\pi).\]
Note that ${\rm Inn}(K)\subset {\rm Aut}(\pi,\rho)$. The image of ${\rm Aut}(\pi,\rho)$ in ${\rm Out}(\pi)$ 
will be denoted as ${\rm Out}(\pi,\rho)$. However ${\rm Inn}(\pi)$ does not preserve $\rho$.   
In order to fix this problem consider the following quotient: 
\[  \widetilde{{\rm Out}(\pi,\rho)}=  {\rm Aut}(\pi,\rho)/{\rm Inn}(K).\]
Then $\widetilde{{\rm Out}(\pi,\rho)}$ has a well-defined action on $ \Mod_{\pi,G}$ and keeps the class $[\rho]$ invariant. 
Note that we have an exact sequence:
\[ 1\to F \to  \widetilde{{\rm Out}(\pi_g,\rho)} \to {\rm Out}(\pi,\rho)\to 1.\]

For any homomorphism $r:F\to G$ the group ${\rm Aut}(\pi,r\circ \rho)$ fixes $r\circ \rho\in {\rm Hom}(\pi,G)$. 
Therefore there is an induced action at the level of Zariski tangent spaces. This  
provides a linear representation of ${\rm Aut}(\pi_g,r\circ \rho)$ on the Zariski tangent space 
$T_{\rho}{\rm Hom}(\pi,G)$ at $r\circ \rho$, which will be called the {\em tangent representation}  at $r\circ \rho$. 
Recall that Weil identified $T_{r\circ \rho}{\rm Hom}(\pi,G)$ with the space of twisted  $1$-cocycles 
$Z^1(\pi, \mathfrak g_{Ad\; r\circ\rho})$ with coefficients in the Lie algebra $\mathfrak g$ twisted by 
the composition of the adjoint representation $Ad$ of $G$ with $r\circ \rho$. This linear 
representation 
\[ {\rm Aut}(\pi,r\circ \rho)\to GL(Z^1(\pi, \mathfrak g_{Ad\; r\circ\rho}))\]
could be defined directly at the level of twisted cocycles $\psi: \pi\to \mathfrak g_{Ad\; r\circ\rho}$, as 
a right composition.

We explained above that $\widetilde{{\rm Out}}(\pi,r\circ\rho)$ acts on $\Mod_{\pi,G}$ and stabilizes 
the class $[r\circ \rho]$ of $r\circ \rho$. We derived then a linear action 
 of $\widetilde{{\rm Out}}(\pi,r\circ\rho)$ on the Zariski tangent space $T_{[\rho]}\Mod_{\pi,G}$.    
 By Weil, this amounts to a linear representation:
\[  \widetilde{{\rm Out}}(\pi,r\circ \rho)\to GL(H^1(\pi, \mathfrak g_{Ad\; r\circ\rho})).\] 

For non-reductive $G$, for instance when $G=\C^*$, we have to modify 
slightly this setting, as it will be explained below. 

This setting also extends to families of representations using intermediary quotients. 
Let us consider the map $\iota_{F}:{\rm Hom}(F,G)\to {\rm Hom}(\pi,G)$, given by $\iota_F(r)=r\circ \rho$. 
We denote by $V_F=\iota_F({\rm Hom}(F,G))\subset  {\rm Hom}(\pi,G)$ the closed subset 
consisting of all those $\rho$ with $\rho(\pi_g)$ isomorphic to a quotient of $F$. 
For any homomorphism $r:F\to G$ we have  ${\rm Aut}(\pi_g,r\circ \rho)\subset {\rm Aut}(\pi_g,\rho)$. 
The group action of ${\rm Aut}(\pi_g,\rho)$ on ${\rm Hom}(\pi_g,G)$ keeps  globally invariant the subvariety $V_F$. 
Note that $V_F$ is not pointwise invariant. 
Consider the Gunning sheaf $TV_F=\cup_{\rho\in V_F} T_{\rho}{\rm Hom}(\pi_g,G)$. As an immediate consequence 
${\rm Aut}(\pi_g,\rho)$ acts both on $TV_F$ and the pull-back $\iota_F^*TV_F$
\[ {\rm Aut}(\pi_g,\rho) \times \iota_F^*TV_F \to \iota_F^*TV_F.\]
We have a similar action $\iota_F:\Mod_{F,G}\to \Mod_{\pi,G}$  whose image $\iota_F(\Mod_{F,G})$ is endowed with 
a Gunning sheaf $T\Mod_{F,G}=\cup_{\rho\in \Mod_{F,G}} T_{\rho}\Mod_{\pi_g,G}$ and a fiber-preserving action: 
\[ \widetilde{{\rm Out}}(\pi_g,\rho) \times \iota_F^*T\Mod_{F,G} \to \iota_F^*T\Mod_{F,G}.\]
We ignored the fact that dimensions of the fibers could be of non-constant dimension.
If we restrict to the non-singular locus of the varieties $\Mod_{F,G}$ or $V_F$, then Gunning sheaves restrict to fiber bundles. 
On any open contractible (in the usual topology)  non-singular subset $U\subset \Mod_{F,G}$ or $V_F$ respectively 
we obtain   linear representations 
\[ U\times {\rm Aut}(\pi,\rho)\to GL(Z^1(\pi, \mathfrak g_{Ad\; r\circ\rho}))\]
and 
\[ U\times  \widetilde{{\rm Out}}(\pi, \rho)\to GL(H^1(\pi, \mathfrak g_{Ad\; r\circ\rho}))\] 
respectively, parameterized by $U$.

\subsection{Finite representations}
The group ${\rm Aut}^+(\pi_g)$ of orientation-preserving automorphisms  of 
$\pi_g$ acts on ${\rm Hom}(\pi_g,G)$ by right composition and this action
passes to a quotient action of $\G_g$ on $\Mod_{g,G}$. 

Let now $F$ be a finite quotient of $\rho:\pi_g\to F$. 
The subgroup  ${\rm Aut}^+(\pi_g, \rho)$  is of finite index in ${\rm Aut}^+(\pi_g)$.
If we fix an embedding $F\subset G$ then 
${\rm Aut}^+(\pi_g,\rho)$ is the stabilizer of 
$\rho$ on ${\rm Hom}(\pi_g,G)$. 
Its image $\G_g(\rho)$ in $\G_g$ is also 
the stabilizer of the class $[\rho]$ of $\rho$ in $\Mod_{g,G}$.

Consider the exact sequence associated to $\rho$: 
\[ 1 \to K\to \pi_g \to F \to 1\]
where $F$ is finite. We are given a representation $r:F\to GL(V)$ 
which induces the structure of $\pi_g$-module on $V$. 
Without loss of generality we can suppose 
that $V$ is from now on an {\em irreducible} 
$F$-module. For the sake of simplicity we consider 
first that $V$ is a complex vector space.

Following \cite{GLLM} we call $\rho$ {\em redundant} if 
it factors through $\pi_g\to \mathbb F_g$ and if  the 
kernel of the homomorphism $\mathbb F_g\to F$ 
contains a free generator. Here $\mathbb F_g$ is the free group of $g$ generators and the homomorphism $\pi_g\to \mathbb F_g$ can be taken as 
the one induced by the inclusion of the surface $\Sigma_g$ 
as the boundary of a handlebody with $g$ handles. 

Furthermore $F\subset G$ is {\em adjoint} if 
the composition $F\to GL(\mathfrak g)$ by the 
adjoint representation $Ad:G\to GL(\mathfrak g)$ is 
an irreducible representation.

\begin{proposition}
Suppose that $\rho$ is a finite adjoint redundant representation of $\pi_g$. 
Then the tangent action at $T_{[\rho]}\Mod_{g,G}$ is an arithmetic group 
of symplectic/orthogonal or linear type. 
\end{proposition}

This is a consequence of the main result of \cite{GLLM}. 
Specifically, one decomposes the 
semisimple algebra $\Q[F]$ into simple algebras:
\[ \Q[F]=\Q\oplus \bigoplus_{i=1}^p A_i\]
where each $A_i$ is a ring of matrices $m_i\times m_i$ over a division algebra 
$D_i$ with center  a number field $L_i$. Each $A_i$ corresponds to a nontrivial irreducible $\Q$-representation of $F$.  
Then the authors of \cite{GLLM} constructed representations 
of (a finite-index subgroup of) $\G_g(\rho)$ into 
the algebraic group of $V_i$-automorphisms 
${\rm Aut}_{A_i}(A_i^{2g-2}, \langle -,- \rangle)$ 
of $A_i^{2g-2}$ endowed with a skew-Hermitian sesquilinear 
$A_i$-valued form. Then the image of this representation     
is a finite index subgroup of 
the arithmetic group ${\rm Aut}_{\mathfrak D_i}(\mathfrak D_i^{2g-2})$, 
where $\mathfrak D_i\subset A_i$ is the  image of $\Z[F]$ 
by the projection onto $A_i$ and is an order in $A_i$.

\begin{proposition}
Assume that $V$ is a nontrivial $F$-module. 
Then we have an isomorphism
\[ H^1(\pi_g, V)\to {\rm Hom}_{\C[F]}(V,V)^{(2g-2)\dim V}. \]
\end{proposition}
\begin{proof}
The  five-term exact sequence reads:
\[ H^1(F, V^K)\to H^1(\pi_g, V)\to H^1(K,V)^F\to H^2(F, V^K).\]
As in the proof of lemma \ref{vanishing} above we use the vanishing of the higher  cohomology 
of a finite group with coefficients in a  $\Q$-vector space
(Prop. 2.1 of \cite{Brown}), in order to derive   
that the restriction homomorphism 
$H^1(\pi_g, V)\to H^1(K,V)^F$ is an isomorphism.

A classical result from \cite{CW} gives a description 
of the $F$-module $H_1(K;\Q)$. Another proof 
is given in \cite{GLLM}. In the case when $\pi_g$ 
were replaced by a free group this was a classical result by Gasch\'utz. 
Specifically, for every $g\geq 2$ we have an isomorphism 
of $F$-modules:
\[ H_1(K;\Q)\to \Q^2\oplus \Q[F]^{2g-2}.\]

Some remarks are in order to understand the action of $F$ on the 
module $H^1(K,V)$. Indeed $F$ acts on $K$ by conjugacy and on $V$ 
through $\rho$. Classes in  $H^1(K,V)$ are
represented by homomorphisms $f:K\to V$, since $V$ is a trivial $K$-module, 
and for $\gamma\in F$, $x\in K$ we have:
 \[ \gamma\cdot f (x)= \rho(\gamma) f(\tilde \gamma^{-1} x\tilde \gamma)\]
where $\tilde \gamma\in \pi_g$ is an arbitrary lift of $\gamma$. 
In particular  the class of $f$ is $F$-invariant if for any 
$\gamma\in F$ and $x\in K$ we have:
\[ f(\tilde \gamma x\tilde \gamma^{-1})=\rho(\gamma) f(x)\]

By the previous description of the $F$-action on 
 $H^1(K,V)$ and the Chevalley-Weil description of 
$H^1(K;\C)$ we derive an isomorphism
\[ H^1(\pi_g, V)\to {\rm Hom}_{\C[F]}(\C[F]^{2g-2}\oplus \C^2, V).\]
On the other hand, for simple $\C[F]$-modules $V$ and $W$ we have   
${\rm Hom}_{\C[F]}(W, V)=0$, unless $V$ and $W$ are 
isomorphic, from Schur's lemma. As a consequence of Maschke's theorem 
$\C[F]=\C\oplus\bigoplus_{i=1}^m V_i^{\dim(V_i)}$, where $V_i$ are all irreducible 
$C[F]$-modules. 
It follows that  
\[{\rm Hom}_{\C[F]}(\C[F]^{2g-2}\oplus \C^2, V)=
{\rm Hom}_{\C[F]}(V,V)^{(2g-2)\dim V}.\]
\end{proof}

Now we have an action of ${\rm Aut}^+(\pi_g,\rho)$ 
on $H^1(\pi_g,V)$  induced by the left composition, which 
we denote by $\phi: {\rm Aut}^+(\pi_g,\rho)\to GL(H^1(\pi_g,V))$. 
Notice however that inner automorphisms do not 
necessarily act trivially. First, not all inner automorphisms 
are in  ${\rm Aut}^+(\pi_g,\rho)$. Second, if the 
conjugacy $\iota_{\alpha}$ by $\alpha\in \pi_g$ does 
belong to ${\rm Aut}^+(\pi_g,\rho)$, then its image is 
the automorphism:
\[ \phi(\iota_{\alpha})=r\rho(\alpha)\]
Since elements in ${\rm Aut}^+(\pi_g,\rho)$ 
which project onto the same element of $\G_g(\rho)$ differ 
by an inner automorphism from ${\rm Aut}^+(\pi_g,\rho)$, it follows that 
we have an induced representation into a quotient group: 
\[ \Phi: \G_g(\rho)\to GL(H^1(\pi_g,V))/r(F).\] 
This is particularly simple when $F$ is abelian, since 
$r(F)$ must be a group of scalar matrices 
and so we obtain a projective representation. 
In the case considered by \cite{GLLM} the authors rather considered 
punctured surfaces in order to work directly with the 
mapping class group $\G_g^1\subset \rm Aut^+(\pi_g)$. 
We have an exact sequence
\[ 1\to \pi_g\to \G_g^1\to \G_g\to 1\]
and the representation $\Phi$ lifts to 
\[ \Phi: \G_g^1(\rho)\to GL(H^1(\pi_g,V)).\] 
The argument from (\cite{GLLM}, section 8.2) shows that 
its restriction to a suitable finite-index subgroup of $\G_g^1(\rho)$
factors through  $\G_g$, so that $\Phi$ lifts to 
a genuine representation after restriction to a finite index subgroup of 
$\G_g(\rho)$.

The case when $F$ is an abelian group and $V$ a 1-dimensional irreducible 
representation of it has been considered by 
Looijenga in \cite{Loo} where the associated representations are called 
Prym representations. This has to be connected with previous 
construction by Gunning (see \cite{Gun}) in genus 2 and 
later extended by Chueshev (see \cite{Chue})
to all genera, which is based on Prym differentials.   

\subsection{Magnus representations for free groups}\label{magnus}
In the case when $\pi=\mathbb F_n$ is a free group, there exists a simple description of these representations. Specifically, we first consider $V=\Z[\mathbb F_n]$ as a left $\mathbb F_n$-module. Then  
\[ H^1(\mathbb F_n, \Z[\mathbb F_n])=I(\mathbb F_n)= \ker(\Z[\mathbb F_n]\to \Z)\]
On the other hand we have an isomorphism 
\[ I(\mathbb F_n) \to (\Z[\mathbb F_n])^n.\]
given by the Fox derivatives. Specifically, if  the $x_i$ form a free basis of $\mathbb F_n$ then 
we send $x\in \mathbb F_n$ into $(\frac{\partial x^{-1}}{\partial x_i})_{i=1,n}$, where the Fox derivatives 
 $\frac{\partial }{\partial x_i}: \mathbb F_n \to \Z[\mathbb F_n]$ form a basis of the space of 1-cocycles and they are determined by: 
 \[ \frac{\partial x_j}{\partial x_i}=\delta_{ij}.\]
Now any automorphism $\varphi$ of $\mathbb F_n$ induces an automorphism of $I(\mathbb F_n)$; under the previous 
isomorphism this automorphism is described as an element of $GL(n, \Z[\mathbb F_n])\subset GL(V^{\oplus n})$ and is given by the matrix 
\[ \overline{\left(\frac{\partial \varphi(x_i)}{\partial x_i}\right)}\in GL(n, \Z[\mathbb F_n])\]
where $\overline{A}$ is the involution of $\Z[\mathbb F_n]$ sending each $x\in \mathbb F_n$ into $x^{-1}$. 

In particular, given a surjective homomorphism $\rho:\mathbb F_n \to F$ we derive a representation 
\[ {\rm Aut}(\mathbb F_n, \rho) \to GL(H^1(\mathbb F_n, \Z[F]))\]
which is obtained from the Magnus representation in $GL(n, \Z[\mathbb F_n])$ by evaluating each entry  via 
$\rho:\Z[\mathbb F_n]\to \Z[F]$. A similar description holds when we choose a family  $V_F$ of representations  
$r:F\to GL(V)$, in which case the tangent representation 
\[ {\rm Aut}(\mathbb F_n, \rho) \to GL(H^1(\mathbb F_n, V_{r\circ \rho}))\]
is obtained by evaluating the Magnus representation entries at points of $V_F$.

\section{Long-Moody twisted cohomological induction}
\subsection{The construction}
Long  and Moody considered in \cite{Long3} a very general recipe for constructing braid group representations (see also \cite{BT}). 
We generalize their construction here to general automorphisms groups. 

{\em Data}. Let $\pi$ be a group, in our case it will be a closed surface group or a free group. 
Let now $B$ be a group related to the automorphisms group ${\rm Aut}(\pi)$, in the sense that it is 
endowed with a homomorphism $\tau: B \to {\rm Aut}(\pi)$. 

Our data  consists of a (finite dimensional){\em  $B$-equivariant linear representation}, namely $\rho:\pi \to GL(V)$ coming  along with a
linear representation  $\beta: B\to GL(V)$ such that $\rho$ is equivariant with respect to the  source and target actions 
$\tau$ and $\beta$: 
\[ \beta(b) \rho(f) = \rho(\tau(b)f) \beta(b), {\rm for \; any} \; b\in B, f\in \pi\]

{\em (Equivariant) twisted cohomological induction}. To every  $B$-equivariant representation:
\[(\rho:\pi \to GL(V), \beta: B\to GL(V), \tau: B \to {\rm Aut}(\pi))\]
we can associate a new representation 
\[ \beta^+: B \to GL(V^+), \; {\rm where } \; V^+=H^1_{\rho}(\pi,V)\]
by the explicit formula: 
\[ (\beta^+(b) \psi) (f) = \beta(b) \left( \psi(\tau^{-1}(b)(f))\right)\]
for every $\psi \in Z^1_{\rho}(\pi, V), f\in \pi, b\in B$. 

\begin{proposition}
The twisted cohomological induction is well-defined. 
\end{proposition}
\begin{proof}
We first have to verify that $\beta^+(b) \psi\in Z^1_{\rho}(\pi,V)$: 
\begin{eqnarray*}
(\beta^+(b) \psi) (fg) & = & \beta(b) \left( \psi(\tau^{-1}(b)(fg))\right)=
\beta(b) \left( \psi(\tau^{-1}(b)(f)\cdot \tau^{-1}(b)(g))\right)=\\
& = & \beta(b) \left( \psi(\tau^{-1}(b)(f) + \rho(\tau^{-1}(b)f) \psi(\tau^{-1}(b)(g)\right)=\\
& = & \beta^+(b)\psi(f) + \beta(b)\rho(\tau^{-1}(b)f) \psi(\tau^{-1}(b)(g)=\\
& = & \beta^+(b)\psi(f) +\rho(f)\beta(b)  \psi(\tau^{-1}(b)(g)=\beta^+(b)\psi(f) +\beta^+(b) \psi(g).\\
\end{eqnarray*}
Moreover this representation on $Z^1_{\rho}(\pi,V)$ descends to $H^1_{\rho}(\pi,V)$. Indeed, if $\psi\in B^1_{\rho}(\pi, V)$, say 
$\psi(g)=\rho(g)v-v$, for any $g\in \pi$ for some $v\in V$, then 
\begin{eqnarray*}
(\beta^+(b) \psi) (g) &= & \beta(b) \left( \psi(\tau^{-1}(b)(g))\right)=
\beta(b)(\rho(\tau^{-1}(b)g)v-v)=\\
& = & \beta(b) \rho(\tau^{-1}(b)(g))v - \beta(b) v=
\rho(g)\beta(b)v-\beta(b)v\in B^1_{\rho}(\pi,V).
\end{eqnarray*}
\end{proof}

\begin{lemma}
A pair $(\rho: \pi \to GL(V), \;\beta:B\to GL(V))$ satisfying the $B$-equivariance is equivalent to 
a representation ${\bm  \beta}: \pi \rtimes_{\tau} B\to GL(V)$ of the semi-direct product group $\pi\rtimes_{\tau} B$ obtained by using the action of $B$ on $\pi$  by means of $\tau$. 
\end{lemma}
\begin{proof} Indeed $\bm{\beta}|_{\pi}=\rho$, while 
$\bm{\beta}|_{s(B)}=\beta$, where $s:B\to \pi \rtimes_{\tau} B$ is a section of the split extension. 
\end{proof}

\begin{remark}
If $\pi$ is either a free group or a surface group and $\beta$ is unitary then, generically $\beta^+$ is unitary (see \cite{Long3}, Thm. 2.8).  
\end{remark}

A linear representation is cohomological if it can be obtained by iterated Long-Moody induction from the trivial representation. 

\begin{question}
It is true that any quantum representation of the mapping class group $\Gamma_{g,1}$, $g\geq 3$, is a subrepresentation of a cohomological representation?
\end{question} 
Here by a quantum representation we mean a representation obtained from a modular tensor category with zero anomaly, e.g. 
obtained from the Turaev--Viro construction. 

\begin{proposition}
The Fibonacci representation $\rho_{2,5}$ of $\Gamma_2$ can be obtained from the trivial representation of the braid groups by cohomological induction. 
\end{proposition}
\begin{proof}
Indeed from (\cite{Long3}, Cor.2.10) we know that all Jones representations
of Hecke algebras associated to Young diagrams with two rows can be obtained by 
cohomological induction. Our description of $\rho_{2,5}$ as a representation of the braid group $B_6$ from Section \ref{sect:fibonacci} 
completes the proof of the claim. 
\end{proof}

\subsection{Examples of cohomological representations}
{\em Braid group representations}. 
Long and Moody used this method to define from a series of representations  $\rho_n:B_n\to GL(V_n)$ of the braid groups 
$B_n$  a new series of linear representations $\rho_{n+1}^+: B_n\to GL(V_{n+1}^n)$ (see \cite{Long3}, Thm.2.1).
 Note the shift in the subscript.
We identify $B_n$ and the mapping class group of the 2-disk with $n$ punctures. The stabilizer of the (first) puncture is 
isomorphic to the semi-direct product $F_n\rtimes_{\tau} B_n\subset B_{n+1}$, where $\tau$ denotes the  Artin representation 
$\tau: B_n\to {\rm Aut}(F_n)$. Then twisted cohomological induction yields a representation 
$\rho_{n+1}^+: B_n\to GL(H^1_{\rho_{n+1}}(\pi,V_{n+1}))$. As $\pi$ is the free group on $n$ generators, 
the standard  free resolution reads (see \cite{Brown}, I.4.4, IV.2, ex.3): 
\[ 0 \to \Z[\pi]^n\to \Z[\pi]\to \Z\to 0.\]
Therefore $H^1_{\rho}(\pi,V)$ is isomorphic to $V^{\oplus n}$. 
With this identification at hand one could write explicitly $\beta^+$ in terms of generators and the values of $\beta$ (see \cite{Long3}, Thm.2.2). 

It is already noticed that there are several embeddings of some semi-direct product $\pi\rtimes B_n$ within $B_{n+1}$. 
Above we considered the pure braid local system in which $\pi$ is freely generated by 
$g_1=\sigma_1^2$, $g_2=\sigma_2 \sigma_1^2 \sigma_2^{-1}$, $g_3=\sigma_3\sigma_2 \sigma_1^2 \sigma_2^{-1}\sigma_3^{-1}$, $\ldots$, $g_n=\sigma_n \sigma_{n-1}\cdots\sigma_2\sigma_1^2
\sigma_2^{-1}\cdots \sigma_{n-1}^{-1}\sigma_n^{-1}$. The action of $B_n$, which is generated by $\sigma_2,\sigma_3, \ldots, \sigma_n$ normalizes the subgroup $\pi$, and the conjugacy action is identified to the action of $B_n$ on the fundamental group 
$\pi$ of the punctured disk. 

If we set $g_1=(\sigma_2\sigma_3\cdots\sigma_n)^n$ and then inductively $g_{i+1}=\sigma_i g_i \sigma_i^{-1}$ then 
the subgroup $\pi$ generated by $g_1,g_2,\ldots,g_n$ is also free of rank $n$ and the subgroup 
$B_n$ generated by $\sigma_1,\sigma_2,\ldots,\sigma_{n-1}$ also normalizes $\pi$. This provides the inner automorphism local system $\pi\rtimes B_n$. Moreover, as we have an obvious map 
$p:\pi\rtimes B_n\to \Z\rtimes B_n$, we can use an arbitrary representation $\beta_n:B_n\to GL(V_n)$ 
and consider $(\beta_n\circ p)^+:B_n\to GL(V_n^{\oplus n})$.

{\em Mapping class group representations}.  According to Lemma \ref{split}, 
$\Gamma_{g,1}^{r |1}=\pi_{g,1}^r\rtimes \Gamma_{g,1}^r$. 
The Long-Moody twisted cohomological induction machinery provides then for 
any representation $\beta: \Gamma_{g,1}^{r+1}\to GL(V)$ another linear representation 
\[ \beta^+: \Gamma_{g,1}^r \to GL(V^{\oplus 2g+r}).\]

{\em Finite index subgroups of mapping class groups}. 
Consider a homomorphism $\rho: \pi \to GL(V)$ and  $B={\rm Aut}^+(\pi,\rho)$, with the usual action on $\pi$ and the 
trivial action $\beta$ on $V$. Then $\beta^+$ is the tangent action of $B$ on ${\rm Hom}(\pi,GL(V))$ at $\rho$. 

{\em Surface braid groups}. We can consider the braid group $B(\Sigma_{g,1},r)=\ker(\Gamma_{g,1}^r\to \Gamma_{g,1})$ on the surface $\Sigma_{g,1}$ on $r$ strands. The isomorphism from Lemma \ref{split}  provides an isomorphism between 
the stabilizer of the last strand in $B(\Sigma_{g,1},r+1)$ and the semi-direct product $\pi_{g,1}^r \rtimes B(\Sigma_{g,1}^r)$.

{\em Magnus representations of ${\rm Aut}(\mathbb F_n)$}. 
In the case when $\pi=\mathbb F_n$ and $\rho:\pi\to F$ has a characteristic kernel $K$, Magnus constructed a 
crossed-homomorphism ${\rm Aut}(\mathbb F_n)\to GL(n, \Z[F])$ whose restriction 
\[ {\rm Aut}(\mathbb F_n, \rho) \to GL(n, \Z[F])\]
is the homomorphism described in section \ref{magnus} (see \cite{Sakasai}). Note that Magnus' homomorphism coincides with the morphism $\beta^+$ provided by 
the construction above to the data $(\rho, \beta)$, where $\beta$ is the left action of ${\rm Aut}(\mathbb F_n)$ on $V=\Z[F]$, after 
identifying $GL(n,\Z[F])$ with a subgroup of $GL(V^{\oplus n})$. 
According to (\cite{Sakasai}, Prop. 3.4)
\[ \ker \beta^+=\ker \left({\rm Aut}(\mathbb F_n)\to {\rm Aut}\left(\frac{\mathbb F_n}{[K,K]}\right)\right).\]

\vspace{0.3cm}
{\bf Acknowledgements}. I'm grateful to J. Aramayona, M. Boggi, P. Eyssidieux, J. March\'e, 
A. Papadopoulos and M. Vazirani for useful discussions.

{
\small      
      
\bibliographystyle{plain}

}

\end{document}